\documentclass[12pt]{article}
\usepackage{times}
\usepackage{lmodern}        % 拉丁现代字体，Computer Modern 的增强版
\usepackage{color}
\usepackage{amsfonts}
\usepackage [dvips]{graphics}
\usepackage[centertags]{amsmath}
\usepackage{amssymb}
\usepackage{amsthm}
\usepackage{newlfont}
\usepackage{mathrsfs}
\usepackage{hyperref}
\usepackage{cases}
\usepackage{amsmath}
\usepackage{bm}
\usepackage{braket}
\usepackage{array}
\usepackage{stmaryrd}
\usepackage{enumitem}
\usepackage{authblk}
\usepackage{breqn}
\newlength{\defbaselineskip}
\newcommand{\setlinespacing}[1]%
           {\setlength{\baselineskip}{#1 \defbaselineskip}}

\theoremstyle{plain}
\newtheorem{thm}{Theorem}[section]
\newtheorem{cor}[thm]{Corollary}
\newtheorem{pro}[thm]{Problem }
\newtheorem{lem}[thm]{Lemma}
\newtheorem{prop}[thm]{Proposition}
\newtheorem{exam}[thm]{Example}

\theoremstyle{definition}
\newtheorem{defn}{Definition}[section]
\newtheorem{ass}{Assumption}[section]
\newtheorem{rmk}{Remark}[section]

\renewcommand{\d}{\mathrm{d}}

\makeatletter\@addtoreset{equation}{section} \makeatother
\DeclareMathOperator*{\esssup}{ess\,sup}

\allowdisplaybreaks

\begin{document}

\title{The Ergodic Linear-Quadratic Optimal Control Problems with Random Periodic Coefficients \thanks{This paper is supported by National Key R\&D Program of China (No.2022YFA1006101), National Natural Science Foundation of China (No.12371445) and State Key Laboratory of Cryptography and Digital Economy Security, Shandong University (No.KFZD2505).}}

\date{}

\author{Jiacheng Wu {\rm $^{a}$}, \hspace{1cm}  Qi  Zhang {\rm $^{a,b,c,}$}\thanks{Corresponding author. Email: qzh@fudan.edu.cn} \hspace{1cm}\\
\small{$^{a}$School of Mathematical Sciences, Fudan University, Shanghai 200433, China}\\
\small{$^{b}$Laboratory of Mathematics for Nonlinear Science, Fudan University, Shanghai 200433, China}\\
\small{$^{c}$State Key Laboratory of Cryptography and Digital Economy Security, Shandong University, Jinan 250100, China}
}
%\small{Email: 22110180042@m.fudan.edu.cn,\ \ \ qzh@fudan.edu.cn}}

\maketitle

\begin{abstract}
	In this paper, we concern with the ergodic linear-quadratic closed-loop optimal control problems with random periodic coefficients. We put forward the random periodic mean-square exponentially stable condition, and prove the random periodicity of solutions to state equation based on it. Then we prove the existence and uniqueness of random periodic solutions to two types of backward stochastic differential equations which serve as stochastic Riccati equations in the procedure of completing the square. With the random periodicity of state equation and stochastic Riccati equations, the ergodic cost functional on infinite horizon is simplified to an equivalent cost functional over a single periodic interval without limit. Finally, the closed-loop optimal controls are explicitly given based on random periodic solutions to state equation and stochastic Riccati equations.
\end{abstract}
\vspace{2mm}

%\begin{AMS}
%93E20, 49K45, 49L20, 49N10
%\end{AMS}
\section{Introduction}

The stochastic linear-quadratic (LQ) optimal control problem plays a pivotal role in the theoretical study and real world. Its power lies in providing tractable, interpretable, linear feedback laws for optimal decision-making under uncertainty across engineering, economics, and finance. For example, the mean-variance model is the core problem in finance which is actually a stochastic LQ problem with the wealth progress satisfying a linear stochastic differential equation (SDE) and the quadratic cost functional given by the variance of wealth. Moreover, in the machine learning, the linear state equation and quadratic cost functional define the foundational structure of the LQ Regulator problem--a canonical benchmark and subclass of Markov decision processes. It can be said without hesitation that the stochastic LQ optimal control problem directly supports critical applications ranging from navigation systems and manufacturing control to dynamic portfolio optimization and macroeconomic policy design. In addition, it forms the essential foundation for modern algorithms in reinforcement learning and adaptive control.

The stochastic LQ problem has undergone a long development phase. Initial results were all derived under the assumption that the coefficients of the control system were deterministic, with the earliest literature traceable to Woham \cite{wo2} in 1968. However, as we all know, stochastic LQ problems with random coefficients are necessary in some situation. Taking the financial model as an example again, when the price dynamics of risky asset depends on historical information, which is reasonable and realistic, the state equation satisfied by the wealth process becomes a SDE with random coefficients.

However, the solvability of stochastic LQ problem with random coefficients has essential difficulties. As Bismut pointed out in 1976 in \cite{bi}, when the coefficients are random, the Riccati equation in the stochastic LQ problem is not a deterministic ordinary differential equation (ODE) but a backward stochastic differential equation (BSDE), to be specific a non-Lipschitz BSDE. Its solvability was described by Bismut as "a challenging task". The challenging task was completely solved in 2003 by Tang \cite{ta}, in which a viewpoint of stochastic flow is used for the associated stochastic Hamiltonian system based on some refined probabilistic methods and techniques. After that, a series of related works emerged. For instance, Tang \cite{ta2} further gives a novel proof to the existence for Riccati equation appearing in stochastic LQ problem with random coefficients via the dynamic programming principle, Li, Wu and Yu \cite{li-wu-yu} studies the stochastic LQ problem with random jumps and Sun, Xiong and Yong \cite{su-xi-yo} investigates indefinite stochastic LQ problems with random coefficients, to name but a few.

In this paper, we focus on a special class of stochastic LQ problems with random coefficients on infinite time horizon, i.e. the ergodic stochastic LQ closed-loop control problems where their coefficients satisfy random periodic properties. Due to the presence of random periodic coefficients, the closed-loop state equation has  random periodic solutions. We aim to explore the roles which random periodic solutions play in the solvability of optimal control and whether they can endow the optimal controls of the stochastic LQ problem with certain distinctive features. Therefore, it is necessary to integrate the concepts and methods of random dynamical systems into the control problem. The main concept of dynamical system we use in this paper is the random periodic solution which is defined in Zhao and Zheng \cite{zh-zhe} and Feng, Zhao and Zhou \cite{fe-zh-zh}. The random periodic solution generalize the stationary solution by extending the properties of pathwise stationarity from any time to a fixed periodic time. After more than a decade of development, many theoretical and practical results on random periodic solutions have emerged. See Wang \cite{wa}, Feng, Wu and Zhao \cite{fe-wu-zh}, Song, Song and Zhang \cite{so-so-zh}, Feng, Qu and Zhao \cite{fe-qu-zh}, Wu and Yuan \cite{wu-yu} for example.

Actually, in our paper \cite{wu-zh}, we have explored this issue in deterministic coefficients case. We give the explicit solution of the optimal control for the ergodic stochastic LQ problem with deterministic periodic coefficients based on the periodic measure. We prove the existence of the periodic measure for state equation and, by the convergence of measures, transform the standard ergodic cost functional into an equivalent cost functional over one period. As a result, the optimal controls are represented with the help of the deterministic Riccati equations with periodic solutions and periodic measures. For more stochastic control problem with deterministic periodic coefficients see also Sun and Yong \cite{su-yo} and Ma \cite{ma}, besides \cite{wu-zh}. However, just as most cases, there exist essential difficulties in the extension from stochastic LQ problems with deterministic periodic coefficients to those with random periodic coefficients.
These difficulties lie not only in the fact that the Riccati equations evolve from deterministic ODEs with periodic solutions to BSDEs with random periodic solutions, but also in that the method of convergence of measures used in \cite{wu-zh} to derive the equivalent transformation of the cost functional is not applicable to random coefficients. Instead, we need to utilize the pathwise periodic properties by a different method since the periodic measures, the steady properties in the distributional sense, are not enough to get the equivalent transformation. Therefore, the case of random periodic coefficients is by no means a simple extension of \cite{wu-zh}.

To realize our goal, we prove that the state equation and the stochastic Riccati equations have random periodic solutions, and utilize their pathwise periodic properties to cover the gap brought by the failure of convergence of measures in the random coefficients case. This allows us to characterize the optimal control with random periodic coefficients in terms of the random periodic solutions. As far as we know, there is no existing results on stochastic LQ problems with random periodic coefficients, except for Guatteri and Masiero \cite{gu-ma2}, in which the stochastic LQ problem with random stationary coefficients are concerned. Under assumptions such as the stationarity of coefficients, a so-called $\sqrt{S}$-stabilizable condition and the uniform boundedness of solution to the stochastic Riccati equation, this paper shows that the value function of ergodic LQ problem is equal to the value function of a stochastic LQ problem over $[0,1]$, and then the optimal control is given by solving the stochastic LQ problem over $[0,1]$. Note that the Brownian motion and the state equation in their model start from $-\infty$, so their result essentially stems from the pathwise stationary property derived from $-\infty$. In contrast, we only suppose the random periodicity of coefficients in this paper, and all results rely on random periodic mean-square exponentially stabilizable condition and necessary positive definiteness of coefficients. Moreover, the Brownian motion and the state equation in our model start normally at time $0$, and instead a fixed-point method is used to get the pathwise steady properties of state equation with a proper initial value in this paper. In particular, if the coefficients satisfy random periodicity for any period, the random periodic coefficients upgrade to random stationary coefficients. Thus, our results could include the stationary coefficients case, though the model setup is different from \cite{gu-ma2}. Finally, it should be noted that we believe that the positive definiteness of coefficients are potentially relaxed to some indefinite condition, which we will consider in future work.

The rest of this paper is organized as follows. In Section 2, we put forward the random periodic mean-square exponentially stable condition, and prove some useful estimates on concerned stochastic control systems. Section 3 is devoted to the existence and uniqueness of the initial state, with which the solution to the state equation is pathwise random periodic. Then the ergodic LQ optimal control problem is introduced and the existence and uniqueness of solutions to stochastic Riccati equations are derived in Section 4. Finally, we transform it into an equivalent stochastic LQ problem over a single periodic interval, and then present the explicit form of optimal solutions based on the random periodic solutions to state equation and stochastic Riccati equations.

\section{Random periodic mean-square exponentially stable condition}
Let $(\Omega,\mathscr{F},\mathbb{P})$ be a complete probability space, on which a standard $1$-dimensional Brownian motion $\left\{W_t,t\ge 0\right\}$ is defined with $\mathbb{F}=\left\{\mathscr{F}_t\right\}_{t\ge0}$ be the natural filtration of $W$ augmented by the $\mathbb{P}$-null sets in $\mathscr{F}$.

For any $T\in(0,\infty]$ and Euclidean space $E$, define\\
$\bullet$ $L_{\mathscr{F}}^2(\Omega;E)=\{\xi:\Omega\rightarrow E\Big|\xi\ \text{is}\ \mathscr{F}\text{-measurable},\mathbb{E}|\xi|^2<\infty\}$,\\
		$\bullet$ $L_{\mathbb{F}}^2(0,T;E)=\{X:\left[0,T\right]\times\Omega\rightarrow E\Big|X\ \text{is}\ \mathbb{F}\text{-progressively measurable},\ \mathbb{E}\int_0^{T}|X_t|^2\d t<\infty\}$,\\
		$\bullet$ $S_{\mathbb{F}}^2(0,T;E)=\{X:\left[0,T\right]\times\Omega\rightarrow E\Big|X\ \text{is}\ \mathbb{F}\text{-progressively measurable},\ \mathbb{E}\sup\limits_{t\in\left[0,T\right]}|X_t|^2<\infty\}$,\\	
		$\bullet$ $L_{\mathscr{F}}^\infty(\Omega;E)=\{\xi:\Omega\rightarrow E\Big|\xi\ \text{is}\ \mathscr{F}\text{-measurable},\ \esssup\limits_{\omega\in\Omega}|\xi|<\infty\}$,\\
		$\bullet$ $L_{\mathbb{F}}^\infty(0,T;E)=\{X:\left[0,T\right]\times\Omega\rightarrow E\Big|X\ \text{is}\ \mathbb{F}\text{-progressively measurable}, \esssup\limits_{(t,\omega)\in[0,T]\times\Omega}|X_t|<\infty\}$,\\
		$\bullet$ $L_{\mathbb{F}}^{2,loc}(0,\infty;E)= \bigcap\limits_{T>0}L_{\mathbb{F}}^2(0,T;E)$.

Denote by $\mathbb{S}^n$ the space of all $n\times n$ symmetric real matrices, by $\mathbb{S}_{+}^n$ the space of all $n\times n$ positive definite real matrices and by $\bar{\mathbb{S}}_{+}^n$ the space of all $n\times n$ positive semi-definite real matrices. $\left\langle\cdot,\cdot\right\rangle$ denotes the inner product of two matrices in the sense that $\left\langle M,N\right\rangle=\text{tr}\left(M^\top N\right)$ for any $M,N\in\mathbb{R}^{m\times n}$. For any $M,N\in\mathbb{S}^n$, we write $M\ge N$, if $M-N\in\bar{\mathbb{S}}_{+}^n$
For a measurable function $f:[0,\infty)\times\Omega\rightarrow\mathbb{S}^n$, $f$ is called uniformly positive definite if there exists an $\alpha>0$ such that
\begin{equation*}
	\begin{aligned}
		(f_t-\alpha I_n)\in\bar{\mathbb{S}}_{+}^n\ {\text for}\ {\text all}\ t\in[0,\infty)\ {\text a.s.}
	\end{aligned}
	\nonumber
\end{equation*}

Consider the shift operator of the Brownian sample path. %discussed in \cite{zh-zh}.\\
For any $t\ge 0$, let $\theta_t:\Omega\rightarrow\Omega$ be a measurable mapping on $(\Omega,\mathscr{F},\mathbb{P})$ defined by
\begin{equation}\label{shift property of Brownian motion}
	\begin{aligned}
		\theta_t\circ W_s=W_{s+t}-W_t \ {\text for}\ {\text any}\ s\ge 0.
	\end{aligned}
\end{equation}
Then $\theta$ satisfies $\theta_0=I$, $\theta_s\circ\theta_t=\theta_{s+t}$ and $\mathbb{P}\cdot\theta_t^{-1}=\mathbb{P}$.
For any $\mathscr{F}$-measurable $\phi:\Omega\rightarrow E$, denote
\begin{equation*}
	\begin{aligned}
		\theta_t\circ\phi(\omega)=\phi(\theta_t(\omega)),
	\end{aligned}
\end{equation*}
and then we clarify some basic notation.
\begin{defn}\label{wz28}
	Given $\tau>0$, a measurable function $f:[0,\infty)\times\Omega\rightarrow E$ is called $\tau$-random periodic if
	\begin{equation*}
		\begin{aligned}
			f_{t+\tau}=\theta_{\tau}\circ f_t\ \ {\text a.s.}\ \ {\text for}\ {\text any}\ t\geq0.
		\end{aligned}
	\end{equation*}
\end{defn}
\begin{defn}
For a solution to a (backward) stochastic differential equation, if it is $\tau$-random periodic for a given $\tau>0$, it is called a $\tau$-random periodic solution, or random periodic solution for simplicity.
\end{defn}
Then we define a space of $\tau$-random periodic functions
$$\mathscr{B}_\tau(E)=\Big\{f:[0,\infty)\times\Omega\rightarrow E\Big|f\in L_{\mathbb{F}}^\infty(0,\infty;E)\text{ and }\tau\text{-random periodic}\Big\}.$$
We use the brief notation $\mathscr{B}_\tau$ if the value space $E$ is clear in above definitions.

Let $\mathcal{N}$ denote the class of $\mathbb{P}$-null sets in $\mathscr{F}$, denote
\begin{equation*}
	\begin{aligned}
		\mathscr{F}_{s,t}^W=\sigma(W_r-W_s:0\le s\le r\le t)\vee\mathcal{N}\ {\text and}\ \mathscr{F}_{t,\infty}^W=\bigvee\limits_{T\ge t}\mathscr{F}_{t,T}^W.
	\end{aligned}
\end{equation*}
In particular, we write $\mathscr{F}_{0,t}^W=\mathscr{F}_t^W$ for simplicity.
Note that if $f$ is $\tau$-random periodic and $\mathbb{F}$-progressively measurable in $[0,\tau)$, then for any $k\in\mathbb{N}$,
%and $0\leq t<\tau$, $f_{k\tau+t}=\theta_{k\tau}\circ f_t$ is $\mathscr{F}_{k\tau,k\tau+t}^W$-measurable, so
$f$ is $\{\mathscr{F}_{k\tau,t}^W\}_{k\tau\leq t<(k+1)\tau}$-progressively measurable in $[k\tau,(k+1)\tau)$, which is compatible with $f$ is $\mathbb{F}$-progressively measurable.

To deal with the infinite horizon state equation, we introduce the mean-square exponentially stable condition.

\begin{defn} For a given $\tau>0$, assume $A,C\in\mathscr{B}_\tau(\mathbb{R}^{n\times n})$. The 2-tuple of coefficients $[A,C]$ is $\tau$-random periodic mean-square exponentially stable if there exist $\beta,\lambda,\delta>0$ such that the homogeneous system
	\begin{equation}\label{homogeneous system}
		\left\{\begin{aligned}
			\d\Phi_t&=A_t\Phi_t\d t+C_t\Phi_t\d W_t,\ \ \ t\ge 0,\\
			\Phi_0&=I_n,
		\end{aligned}\right.
	\end{equation}
	admits a unique solution $\Phi\in L_{\mathbb{F}}^{2}(0,\infty;\mathbb{R}^{n\times n})$ and for any $t\ge0$,
	\begin{equation}\label{stability condition}
		\begin{aligned}
			&\mathbb{E}\big|\Phi_t\big|^2\le \beta e^{-\lambda t},\\			&\inf\limits_{r\in[0,\tau)}\mathbb{E}\Big(\int_r^\infty(\Phi_s\Phi_r^{-1})^\top\Phi_s\Phi_r^{-1}\d s\Big|\mathscr{F}_r\Big)\ge\delta I_n.
		\end{aligned}
	\end{equation}
\end{defn}
Actually \eqref{stability condition} holds under certain condition. We present an example below.
\begin{exam}
Consider a $1$-dimensional SDE
\begin{equation}\label{1 dim case}
		\left\{\begin{aligned}
			\d\Phi_t=&a_t\Phi_t\d t+c_t\Phi_t\d W_t,\ \ \ t\ge 0,\\
			\Phi_0=&1.
		\end{aligned}\right.
	\end{equation}
For a given $\tau>0$, assume $a,c\in\mathscr{B}_\tau(\mathbb{R})$, and there exist $0<\mu_1<\mu_2$ such that for any $t\ge 0$, $-\mu_2\le2a_t+c_t^2\le-\mu_1$ a.s. Then \eqref{1 dim case} has a unique solution $\Phi\in L_{\mathbb{F}}^{2}(0,\infty;\mathbb{R})$ with an explicit form
	\begin{equation*}
		\begin{aligned}
			\Phi_t=e^{\int_0^t(a_s-\frac{1}{2}c_s^2)\d s+\int_0^tc_s\d W_s},\ \ \ t\ge 0.
		\end{aligned}
	\end{equation*}
Obviously, \eqref{stability condition} holds since
	\begin{equation*}
		\begin{aligned}
			\mathbb{E}\big|\Phi_t\big|^2&=\mathbb{E}e^{\int_0^t(2a_s+c_s^2)\d s}\le e^{-\mu_1 t},\\
			\mathbb{E}\Big(\int_r^\infty(\Phi_s\Phi_r^{-1})^\top\Phi_s\Phi_r^{-1}\d s\Big|\mathscr{F}_r\Big)&=\mathbb{E}\Big(\int_r^\infty e^{\int_r^{s}(2a_u+c_u^2)\d u}\d s\Big|\mathscr{F}_r\Big)\\
			&\ge\mathbb{E}\int_r^\infty e^{-\mu_2(s-r)}\d s=1/\mu_2.
		\end{aligned}
	\end{equation*}
\end{exam}

Based on the independent increment of Brownian motion, we have the following lemma for solutions to the homogeneous closed-loop system (\ref{homogeneous system}).
\begin{lem}\label{time shift property of homogeneous system} For a given $\tau>0$, assume $A,C\in\mathscr{B}_\tau(\mathbb{R}^{n\times n})$ and $\Phi$ is the solution to \eqref{homogeneous system}. Then for any $t\ge 0$, $k\in\mathbb{N}$,
	\begin{equation*}
		\begin{aligned}
			&\Phi_{t+k\tau}\Phi_{k\tau}^{-1}=\theta_{k\tau}\circ\Phi_t \ \ {\text a.s.},\\
			&\Phi_{t+k\tau}\Phi_{k\tau}^{-1} \ {\text and } \ \Phi_{k\tau} \ {\text are} \ {\text independent}.
		\end{aligned}
	\end{equation*}
\end{lem}
\begin{proof}
	Notice that for any $t\ge 0$ ,$k\in\mathbb{N}$, $\Phi_{t+k\tau}\Phi_{k\tau}^{-1}$ satisfies SDE
	\begin{equation*}
		\left\{\begin{aligned}
			\d\phi_t&=A_{t+k\tau}\phi_t\d t+C_{t+k\tau}\phi_t\d W_{t+k\tau},\ \ \ t\ge 0,\\
			\phi_0&=I_n.
		\end{aligned}\right.
	\end{equation*}
	and $\theta_{k\tau}\circ\Phi_t$ satisfies SDE
	\begin{equation*}
		\left\{\begin{aligned}
			\d\phi_t&=(\theta_{k\tau}\circ A_t)\phi_t\d t+(\theta_{k\tau}\circ C_t)\phi_t\d (\theta_{k\tau}\circ W_t),\ \ \ t\ge 0,\\
			\phi_0&=I_n.
		\end{aligned}\right.
	\end{equation*}
	Due to (\ref{shift property of Brownian motion}),
	\begin{equation*}
		\begin{aligned}
			\theta_{k\tau}\circ W_t=W_{t+k\tau}-W_{k\tau}.
		\end{aligned}
	\end{equation*}
	By the $\tau$-random periodicity of $A,C$, together with the measurability of $\Phi_{t+k\tau}\Phi_{k\tau}^{-1}$ and $\theta_{k\tau}\circ\Phi_t$ with respect to $\mathscr{F}_{k\tau,t+k\tau}^W$, we have
	\begin{equation*}
		\begin{aligned}
			\Phi_{t+k\tau}\Phi_{k\tau}^{-1}=\theta_{k\tau}\circ\Phi_t\ \ {\text a.s.}
		\end{aligned}
	\end{equation*}
	
	On the other hand, bear in mind that $\Phi_{t+k\tau}\Phi_{k\tau}^{-1}$ and $\Phi_{k\tau}$ is measurable with respect to $\mathscr{F}_{k\tau,t+k\tau}^W$ and $\mathscr{F}_{k\tau}^W$, respectively.
But $\mathscr{F}_{k\tau,t+k\tau}^W$ and $\mathscr{F}_{k\tau}^W$ are independent due to the independent increment property of Brownian motion. Hence the independence of $\Phi_{t+k\tau}\Phi_{k\tau}^{-1}$ and $\Phi_{k\tau}$ follows.
\end{proof}

In the following two propositions, we give an equivalent characterization of the random periodic mean-square exponentially stable condition.
\begin{prop}\label{linear BSDE} For a given $\tau>0$, assume that
	$[A,C]$ is $\tau$-random periodic mean-square exponentially stable and $\Lambda\in \mathscr{B}_\tau(\mathbb{S}^n)$. Then BSDE
	\begin{equation}\label{linear BSDE equation}
		\begin{aligned}
			\d K_t=&-\big(K_tA_t+A_t^\top K_t+C_t^\top K_tC_t+L_tC_t+C_t^\top L_t+\Lambda_t\big)\d t+L_t\d W_t
		\end{aligned}
	\end{equation}
admits a unique $\tau$-random periodic solution $(K,L)\in L_{\mathbb{F}}^\infty(0,\infty;\mathbb{S}^n)\times L_{\mathbb{F}}^{2,loc}(0,\infty;\mathbb{S}^n)$.  Moreover, if $\Lambda\in\mathscr{B}_\tau(\bar{\mathbb{S}}_+^n)$, then $K\in\mathscr{B}_\tau(\bar{\mathbb{S}}_+^n)$; if $\Lambda\in\mathscr{B}_\tau(\mathbb{S}_+^n)$ is uniformly positive definite, then $K\in\mathscr{B}_\tau(\mathbb{S}_+^n)$ is uniformly positive definite.
\end{prop}
\begin{proof}
	For any $K^*_{\tau}\in L_{\mathscr{F}_\tau}^\infty(\Omega;\mathbb{S}^n)$, by Lemma 5.2 in Peng \cite{pe}, BSDE (\ref{linear BSDE equation}) admits a unique solution $(K^*,L^*)\in L_{\mathbb{F}}^\infty(0,\tau;\mathbb{S}^n)\times L_{\mathbb{F}}^2(0,\tau;\mathbb{S}^n)$.
	
Let $\Phi$ be the solution to \eqref{homogeneous system}. Applying It\^{o} formula to $\Phi_s^\top K^*_s\Phi_s$, we have
	\begin{equation*}
		\begin{aligned}
			K^*_0=\mathbb{E}\Big(\int_0^\tau\Phi_s^\top\Lambda_s\Phi_s\d s+\Phi_{\tau}^\top K^*_{\tau}\Phi_{\tau}\Big).
		\end{aligned}
	\end{equation*}
	Our target is to show there exist a unique $M\in\mathbb{R}^{n\times n}$ such that
	\begin{equation*}
		\begin{aligned}
			M=\mathbb{E}\Big(\int_0^\tau\Phi_s^\top\Lambda_s\Phi_s\d s+\Phi_{\tau}^\top M\Phi_{\tau}\Big).
		\end{aligned}
	\end{equation*}
%	Denote
%	\begin{equation*}
%		\begin{aligned}
%			G=\mathbb{E}\Big(\int_0^\tau\Phi_s^\top\Lambda_s\Phi_s\d s\Big).
%		\end{aligned}
%	\end{equation*}
For this, let $M_0=G\triangleq\mathbb{E}\Big(\int_0^\tau\Phi_s^\top\Lambda_s\Phi_s\d s\Big)$, and for any $k\in\mathbb{N}^+$, define
	\begin{equation*}
		\begin{aligned}
			M_k=G+\mathbb{E}\Big(\Phi_{\tau}^\top M_{k-1}\Phi_{\tau}\Big).
		\end{aligned}
	\end{equation*}
	By Lemma \ref{time shift property of homogeneous system}
	\begin{equation}\label{shift property}
		\begin{aligned}
			&\Phi_{2\tau}\Phi_{\tau}^{-1}=\theta_\tau\circ \Phi_{\tau}\ \ {\text a.s.},\\
			&\Phi_{2\tau}\Phi_{\tau}^{-1} \text{ and } \Phi_{\tau} \text{ are independent}.
		\end{aligned}
	\end{equation}
So the measure preserving property of $\theta_\tau$ leads to
	\begin{equation*}
		\begin{aligned}
			M_k=&G+\mathbb{E}\Big(\Phi_{\tau}^\top M_{k-1}\Phi_{\tau}\Big)\\
			=&G+\mathbb{E}\Big(\Phi_{\tau}^\top G\Phi_{\tau}\Big)+\mathbb{E}\Big(\Phi_{\tau}^\top\mathbb{E}\Big(\Phi_{\tau}^\top M_{k-2}\Phi_{\tau}\Big)\Phi_{\tau}\Big)\\
			=&G+\mathbb{E}\Big(\Phi_{\tau}^\top G\Phi_{\tau}\Big)+\mathbb{E}\Big(\Phi_{\tau}^\top\mathbb{E}\Big((\Phi_{2\tau}\Phi_{\tau}^{-1})^\top M_{k-2}\Phi_{2\tau}\Phi_{\tau}^{-1}\Big)\Phi_{\tau}\Big)\\
			=&G+\mathbb{E}\Big(\Phi_{\tau}^\top G\Phi_{\tau}\Big)+\mathbb{E}\Big(\Phi_{2\tau}^\top M_{k-2}\Phi_{2\tau}\Big).
		\end{aligned}
	\end{equation*}
	By induction, we get
$M_N=\sum_{k=0}^{N}\mathbb{E}\Big(\Phi_{k\tau}^\top G\Phi_{k\tau}\Big)$.
	Since $[A,C]$ is $\tau$-random periodic mean-square exponentially stable, there exist $\beta,\lambda>0$ such that for $k\in\mathbb{N}$,
	\begin{equation}\label{stable ktau}
		\begin{aligned}
			\mathbb{E}\Big(\text{tr}(\Phi_{k\tau}^\top\Phi_{k\tau})\Big)\le \beta e^{-\lambda k\tau}.
		\end{aligned}
	\end{equation}
	By the monotone convergence theorem, there exists $M\in \mathbb{R}^{n\times n}$ such that $M_N\rightarrow M$ as $N\rightarrow\infty$. Moreover, if $\Lambda\ge 0$, then $M\ge G\ge0$.
	
	Applying It\^{o} formula to $\Phi_s^\top K^*_s\Phi_s$ again, for any $r\in[0,\tau)$, we have
	\begin{equation*}
		\begin{aligned}
			K^*_r=\mathbb{E}\Big(\int_r^\tau(\Phi_s\Phi_r^{-1})^\top\Lambda_s\Phi_s\Phi_r^{-1}\d s+(\Phi_{\tau}\Phi_r^{-1})^\top M\Phi_{\tau}\Phi_r^{-1}\Big|\mathscr{F}_r\Big).
		\end{aligned}
	\end{equation*}
	Thus for any $r\in[0,\tau)$, $K^*_r\ge0$ a.s.
	
	In general, for any $k\in\mathbb{N}$, $t\in(k\tau,(k+1)\tau]$, define
	\begin{equation*}
		\begin{aligned}
			K_t=\theta_{k\tau}\circ K^*_{t-k\tau},\ \ L_t=\theta_{k\tau}\circ L^*_{t-k\tau}.
		\end{aligned}
	\end{equation*}
Since $A,C,\Lambda$ are $\tau$-random periodic, $(K,L)\in L_{\mathbb{F}}^\infty(0,\infty;\mathbb{S}^n)\times L_{\mathbb{F}}^{2,loc}(0,\infty;\mathbb{S}^n)$ satisfies \eqref{linear BSDE equation}.
	Hence for any $t\ge 0$,
	\begin{equation*}
		\begin{aligned}
			&K_{t+\tau}=\theta_{\tau}\circ K_t,\ \ L_{t+\tau}=\theta_{\tau}\circ L_t,\\
			&K_t\ge 0\ \ {\text a.s.}
		\end{aligned}
	\end{equation*}
	By Lemma \ref{time shift property of homogeneous system}, for any $t\ge 0$, $k\in\mathbb{N}$,
	\begin{equation*}
		\begin{aligned}
			&\Phi_{t+k\tau}\Phi_{k\tau}^{-1}=\theta_{k\tau}\circ\Phi_t\ \ {\text a.s.},\\
			&\Phi_{t+k\tau}\Phi_{k\tau}^{-1} \text{ and } \Phi_{k\tau} \text{ are independent}.
		\end{aligned}
	\end{equation*}
	Thus
	\begin{equation*}
		\begin{aligned}
			M=&\sum_{k=0}^{\infty}\mathbb{E}\Big(\Phi_{k\tau}^\top G\Phi_{k\tau}\Big)\\
			=&\sum_{k=0}^{\infty}\mathbb{E}\Big(\Phi_{k\tau}^\top\mathbb{E}\Big(\int_0^\tau\Phi_s^\top\Lambda_s\Phi_s\d s\Big)\Phi_{k\tau}\Big)\\
			=&\sum_{k=0}^{\infty}\mathbb{E}\Big(\Phi_{k\tau}^\top\mathbb{E}\Big(\int_0^\tau\theta_{k\tau}\circ\Phi_s^\top\theta_{k\tau}\circ\Lambda_s\theta_{k\tau}\circ\Phi_s\d s\Big)\Phi_{k\tau}\Big)\\
			=&\sum_{k=0}^{\infty}\mathbb{E}\Big(\Phi_{k\tau}^\top\mathbb{E}\Big(\int_0^\tau(\Phi_{s+k\tau}\Phi_{k\tau}^{-1})^\top\Lambda_{s+k\tau}\Phi_{s+k\tau}\Phi_{k\tau}^{-1}\d s\Big)\Phi_{k\tau}\Big)\\
			=&\sum_{k=0}^{\infty}\mathbb{E}\Big(\int_0^\tau\Phi_{s+k\tau}^\top\Lambda_{s+k\tau}\Phi_{s+k\tau}\d s\Big)\\
			=&\mathbb{E}\Big(\int_0^{\infty}\Phi_s^\top\Lambda_s\Phi_s\d s\Big).
		\end{aligned}
	\end{equation*}
	Then for any $r\in[0,\tau)$, we have
	\begin{equation}\label{linear BSDE formula}
		\begin{aligned}
			K_r=&\mathbb{E}\Big(\int_r^\tau(\Phi_s\Phi_r^{-1})^\top\Lambda_s\Phi_s\Phi_r^{-1}\d s+(\Phi_\tau\Phi_r^{-1})^\top M\Phi_\tau\Phi_r^{-1}\Big|\mathscr{F}_r\Big)\\
			=&\mathbb{E}\Big(\int_r^\infty(\Phi_s\Phi_r^{-1})^\top\Lambda_s\Phi_s\Phi_r^{-1}\d s\Big|\mathscr{F}_r\Big).
		\end{aligned}
	\end{equation}
	If there exists $\alpha>0$ such that for any $t\ge 0$, $\Lambda_t\ge\alpha I_n$ a.s., then
	\begin{equation*}
		\begin{aligned}
			K_r\ge\alpha\mathbb{E}\Big(\int_r^\infty(\Phi_s\Phi_r^{-1})^\top\Phi_s\Phi_r^{-1}\d s\Big|\mathscr{F}_r\Big)\ge\alpha\delta I_n.
		\end{aligned}
	\end{equation*}
	Thus for any $r\in[0,\tau)$, $k\in\mathbb{N}$, $K_{k\tau+r}=\theta_{k\tau}\circ K_r\ge\alpha\delta I_n$ a.s.,
	which puts the end of the proof.
\end{proof}
\begin{prop}\label{stable condition characterization} For a given $\tau>0$, assume that $A,C\in\mathscr{B}_\tau(\mathbb{R}^{n\times n})$, $\Lambda\in\mathscr{B}_\tau(\mathbb{S}_+^n)$ is  uniformly positive definite, BSDE \eqref{linear BSDE equation} admits a unique $\tau$-random periodic solution $(K,L)\in L_{\mathbb{F}}^\infty(0,\infty;\mathbb{S}_+^n)\times L_{\mathbb{F}}^{2,loc}(0,\infty;\mathbb{S}^n)$, and $K\in\mathscr{B}_\tau(\mathbb{S}_+^n)$ is uniformly positive definite. Then $[A,C]$ is $\tau$-random periodic mean-square exponentially stable.
\end{prop}
\begin{proof}Let $\Phi_t$ be the solution to \eqref{homogeneous system}. Applying It\^{o} formula to $\Phi_t^\top K_t\Phi_t$, we have
	\begin{equation*}
		\begin{aligned}
			\frac{\d}{\d t}\mathbb{E}\Big(\Phi_t^\top K_t\Phi_t\Big)=-\mathbb{E}\Big(\Phi_t^\top \Lambda_t\Phi_t\Big).
		\end{aligned}
	\end{equation*}
	Since $\Lambda, K\in\mathscr{B}_\tau(\mathbb{S}_+^n)$ are uniformly positive definite, there exist $\delta,\alpha>0$, such that for any $t\ge 0$,
	\begin{equation*}
		\begin{aligned}
			\delta I_n\le\Lambda_t,\ \ K_t\le\alpha I_n\ \ {\text a.s.}
		\end{aligned}
	\end{equation*}
	Then
	\begin{equation*}
		\begin{aligned}
			\frac{\d}{\d t}\mathbb{E}\Big(\Phi_t^\top K_t\Phi_t\Big)&\le-\delta\mathbb{E}\Big(\Phi_t^\top\Phi_t\Big)\le-\frac{\delta}{\alpha}\mathbb{E}\Big(\Phi_t^\top K_t\Phi_t\Big).
		\end{aligned}
	\end{equation*}
	By Gronwall inequality, for any $t\ge0$,
	\begin{equation*}
		\begin{aligned}
			\mathbb{E}\text{tr}\Big(\Phi_t^\top\Phi_t\Big)\le\frac{1}{\delta}\mathbb{E}\text{tr}\Big(\Phi_t^\top K_t\Phi_t\Big)\le\frac{\alpha}{\delta}e^{-\frac{\delta}{\alpha}k\tau}.
		\end{aligned}
	\end{equation*}
Noticing \eqref{linear BSDE formula}, we have for any $r\in [0,\tau)$,
	\begin{equation*}
		\begin{aligned}
			\mathbb{E}\Big(\int_r^\infty(\Phi_s\Phi_r^{-1})^\top\Phi_s\Phi_r^{-1}\d s\Big|\mathscr{F}_r\Big)
			\ge&\frac{1}{\alpha}\mathbb{E}\Big(\int_r^\infty(\Phi_s\Phi_r^{-1})^\top\Lambda_s\Phi_s\Phi_r^{-1}\d s\Big|\mathscr{F}_r\Big)\\
			=&\frac{1}{\alpha}K_r
			\ge\frac{\delta}{\alpha}I_n.
		\end{aligned}
	\end{equation*}
The desired results is derived.
\end{proof}

Based on the equivalent characterization of the
random periodic mean-square exponentially stable condition, we further have the well-posedness and periodicity of the solution to the dual equation of closed-loop state equation.
\begin{prop}\label{linear one order BSDE} For a given $\tau>0$, assume
that $[A,C]$ is $\tau$-random periodic mean-square exponentially stable and $(K,L)\in L_{\mathbb{F}}^\infty(0,\infty;\mathbb{S}^n)\times L_{\mathbb{F}}^{2,loc}(0,\infty;\mathbb{S}^n)$ is the solution to \eqref{linear BSDE equation}. Then for any $b,\sigma,\lambda\in\mathscr{B}_\tau(\mathbb{R}^n)$, BSDE
	\begin{equation}\label{linear one order BSDE equation}
		\begin{aligned}
			\d \eta_t=-\big(A_t^\top\eta_t+C_t^\top\zeta_t+K_tb_t+C_t^\top K_t\sigma_t+L_t\sigma_t+\lambda_t\big)\d t+\zeta_t\d W_t
		\end{aligned}
	\end{equation}
admits a unique $\tau$-random periodic solution $(\eta,\zeta)\in S_{\mathbb{F}}^2(0,\infty;\mathbb{R}^n)\times L_{\mathbb{F}}^{2,loc}(0,\infty;\mathbb{R}^n)$.
\end{prop}
\begin{proof}For any $\eta^*_\tau\in L_{\mathscr{F}_\tau}^2(\Omega;\mathbb{R}^n)$, BSDE \eqref{linear one order BSDE equation} has a unique solution $(\eta^*,\zeta^*)\in S_{\mathbb{F}}^2(0,\infty;\mathbb{R}^n)\times L_{\mathbb{F}}^2(0,\infty;\mathbb{R}^n)$.

Let $\Phi_t$ be the solution to \eqref{homogeneous system}. Applying It\^{o} formula to $\Phi_s^\top\eta_s^*$, we have
	\begin{equation*}
		\begin{aligned}
			\eta^*_0=\mathbb{E}\Big(\int_0^\tau\Phi_s^\top(K_sb_s+C_s^\top K_s\sigma_s+L_s\sigma_s+\lambda_s)\d s+\Phi_\tau^\top\eta^*_\tau\Big).
		\end{aligned}
	\end{equation*}
	Our target is to show there exist a unique $g\in\mathbb{R}^n$ such that
	\begin{equation*}
		\begin{aligned}
			g=\mathbb{E}\Big(\int_0^\tau\Phi_s^\top(K_sb_s+C_s^\top K_s\sigma_s+L_s\sigma_s+\lambda_s)\d s+\Phi_\tau^\top g\Big).
		\end{aligned}
	\end{equation*}
For this, let $g_0=h\triangleq\mathbb{E}\Big(\int_0^\tau\Phi_s^\top(K_sb_s+C_s^\top K_s\sigma_s+L_s\sigma_s+\lambda_s)\d s\Big)$, and for any $k\in\mathbb{N}^+$, define
	\begin{equation*}
		\begin{aligned}
			g_k=h+\mathbb{E}\Big(\Phi_\tau^\top g_{k-1}\Big).
		\end{aligned}
	\end{equation*}
It follows from \eqref{shift property} and measure preserving property of $\theta_\tau$ that
	\begin{equation*}
		\begin{aligned}
			g_k=&h+\mathbb{E}\Big(\Phi_\tau^\top g_{k-1}\Big)\\
			=&h+\mathbb{E}\Big(\Phi_\tau^\top h\Big)+\mathbb{E}\Big(\Phi_\tau^\top\mathbb{E}\Big(\Phi_\tau^\top g_{k-2}\Big)\Big)\\
			=&h+\mathbb{E}\Big(\Phi_\tau^\top h\Big)+\mathbb{E}\Big(\Phi_\tau^\top\mathbb{E}\Big((\Phi_{2\tau}\Phi_\tau^{-1})^\top g_{k-2}\Big)\Big)\\
			=&h+\mathbb{E}\Big(\Phi_\tau^\top h\Big)+\mathbb{E}\Big(\Phi_{2\tau}^\top g_{k-2}\Big).
		\end{aligned}
	\end{equation*}
	By induction, we have
$g_N=\sum_{k=0}^{N}\mathbb{E}\Big(\Phi_{k\tau}^\top h\Big)$.
Hence by \eqref{stable ktau}, for any $N_1,N_2\in\mathbb{N}^+$ and $N_1<N_2$, it yields that
	\begin{equation*}
		\begin{aligned}
			\big|g_{N_2}-g_{N_1}\big|^2=&\text{tr}\Big(\mathbb{E}\big(\sum_{k=N_1+1}^{N_2}\Phi_{k\tau}^\top h\big)^\top\mathbb{E}\sum_{k=N_1+1}^{N_2}\Phi_{k\tau}^\top h\Big)\\
			\le&\sum_{k=N_1+1}^{N_2}\text{tr}\Big(hh^\top e^{\frac{\lambda k\tau}{2}}\mathbb{E}\Phi_{k\tau}\mathbb{E}\Phi_{k\tau}^\top\Big)\sum_{k=N_1+1}^{N_2}e^{-\frac{\lambda k\tau}{2}}\\
			\le&\sum_{k=N_1+1}^{N_2}\mathbb{E}\Bigg(\text{tr}\Big(hh^\top e^{\frac{\lambda k\tau}{2}}\Phi_{k\tau}\Phi_{k\tau}^\top\Big)\Bigg)\sum_{k=N_1+1}^{N_2}e^{-\frac{\lambda k\tau}{2}}\\
			\le&c\sum_{k=N_1+1}^{N_2}e^{-\frac{\lambda k\tau}{2}}\sum_{k=N_1+1}^{N_2}e^{-\frac{\lambda k\tau}{2}}\text{tr}\Big(hh^\top\Big)\\
			\le&ce^{-\lambda(N_1+1)\tau}(1-e^{-\frac{\lambda\tau}{2}})^{-2}\text{tr}\Big(hh^\top\Big).
		\end{aligned}
	\end{equation*}
Here and in the rest of this paper, $c$ is a generic constant which may change from line to line.	This shows that $\left\{g_N\right\}_{N\in\mathbb{N}^+}$ is a Cauchy sequence in $\mathbb{R}^n$, and thus there exists a $g\in\mathbb{R}^n$ such that $g_N\rightarrow g$ as $N\rightarrow\infty$.
	
	In general, for any $k\in\mathbb{N}$, $t\in(k\tau,(k+1)\tau]$, define
	\begin{equation*}
		\begin{aligned}
			\eta_t=\theta_{k\tau}\circ \eta^*_{t-k\tau},\zeta_t=\theta_{k\tau}\circ \zeta^*_{t-k\tau}.
		\end{aligned}
	\end{equation*}
Since $A,C,b,\sigma,\lambda,K,L$ are $\tau$-random periodic, $(\eta,\zeta)\in L_{\mathbb{F}}^\infty(0,\infty;\mathbb{R}^n)\times L_{\mathbb{F}}^{2,loc}(0,\infty;\mathbb{R}^n)$ satisfies \eqref{linear BSDE equation}.
Hence for any $t\ge 0$,
	\begin{equation*}
		\begin{aligned}
			&\eta_{t+\tau}=\theta_{\tau}\circ \eta_t,\ \ \zeta_{t+\tau}=\theta_{\tau}\circ \zeta_t.
		\end{aligned}
	\end{equation*}
	The proof is complete.
\end{proof}

\section{Random periodic solution to SDE}

With the help of Propositions \ref{linear BSDE} and \ref{linear one order BSDE}, we have some useful  estimates for a linear SDE.
\begin{prop}\label{priori estimate} For a given $\tau>0$, assume
	that $[A,C]$ is $\tau$-random periodic mean-square exponentially stable and $b,\sigma \in\mathscr{B}_\tau(\mathbb{R}^n)$. Then there exists $\beta,\lambda>0$ such that for any $r\in[0,\tau)$, SDE with the $\mathscr{F}_{r,\infty}^W$-independent initial state $\xi\in L_{\mathscr{F}_r}^2(\Omega;\mathbb{R}^n)$
	\begin{equation}\label{semi linear SDE}
		\left\{\begin{aligned}
			d X_t&=\big(A_tX_t+b_t\big)\d t+\big(C_tX_t+\sigma_t\big)\d W_t,\ \ \ t\ge r,\\
			X_r&=\xi
		\end{aligned}\right.
	\end{equation}
admits a unique solution $X^{r,\xi}\in L_{\mathbb{F}}^{2,loc}(r,\infty;\mathbb{R}^{n})$,	and for any $\mathscr{F}_{r,\infty}^W$-independent $\xi, \xi_1, \xi_2\in L_{\mathscr{F}_r}^2(\Omega;\mathbb{R}^n)$, $k\in\mathbb{N}$,
	\begin{equation}\label{delta square estimate}
		\begin{aligned}
			\mathbb{E}\big|X^{r,\xi_1}_{k\tau+r}-X^{r,\xi_2}_{k\tau+r}\big|^2\le\beta e^{-\lambda k\tau}\mathbb{E}|\xi_1-\xi_2|^2
		\end{aligned}
	\end{equation}
and
	\begin{equation}\label{square estimate}
		\begin{aligned}
			\mathbb{E}\big|X^{r,\xi}_{k\tau+r}\big|^2\le \beta(1+\mathbb{E}|\xi|^2).
		\end{aligned}
	\end{equation}
\end{prop}
\begin{proof} The solvability of SDE \eqref{semi linear SDE} is obvious, and we only need to prove \eqref{delta square estimate} and \eqref{square estimate}. For any $r\in[0,\tau)$, let $X^{r,\xi_1},X^{r,\xi_2}$ be the solutions to \eqref{semi linear SDE} with the initial states $\xi_1,\xi_2$, respectively. Set
	\begin{equation*}
		\begin{aligned}
			\widehat{\xi}\triangleq\xi_1-\xi_2\ \ {\text and}\ \ \widehat{X}_t\triangleq X_t^{r,\xi_1}-X_t^{r,\xi_2},\ \ \ t\geq r.
		\end{aligned}
	\end{equation*}
	Then $\widehat{X}$ satisfies SDE
	\begin{equation*}
		\left\{\begin{aligned}
			\d \widehat{X}_t&=A_t\widehat{X}_t\d t+C_t\widehat{X}_t\d W_t,\ \ \ t\ge r,\\
			\widehat{X}_r&=\widehat{\xi}.
		\end{aligned}\right.
	\end{equation*}
	By Proposition \ref{linear BSDE}, BSDE
	\begin{equation}\label{linear BSDE Lambda=I}
		\begin{aligned}
			\d K_t=&-\big(K_tA_t+A_t^\top K_t+C_t^\top K_tC_t+L_tC_t+C_t^\top L_t+I_n\big)\d t+L_t\d W_t
		\end{aligned}
	\end{equation}
admits a unique $\tau$-random periodic solution $(K,L)\in L_{\mathbb{F}}^\infty(0,\infty;\mathbb{S}^n)\times L_{\mathbb{F}}^{2,loc}(0,\infty;\mathbb{S}^n)$, and $K\in\mathscr{B}_\tau(\mathbb{S}_+^n)$ is uniformly positive definite.	Then there exist $\delta,\alpha>0$, such that for any $t\ge 0$,
	\begin{equation*}
		\begin{aligned}
			\delta I_n\le K_t\le\alpha I_n\ \ {\text a.s.}
		\end{aligned}
	\end{equation*}
Applying It\^{o} formula to $\widehat{X}_t^\top K_t\widehat{X}_t$, we have
	\begin{equation*}
		\begin{aligned}
			\frac{\d}{\d t}\mathbb{E}\Big(\widehat{X}_t^\top K_t\widehat{X}_t\Big)=-\mathbb{E}\Big(\widehat{X}_t^\top\widehat{X}_t\Big)
			\le-\frac{1}{\alpha}\mathbb{E}\Big(\widehat{X}_t^\top K_t\widehat{X}_t\Big).
		\end{aligned}
	\end{equation*}
	By Gronwall inequality, for any $r\in[0,\tau)$, $k\in\mathbb{N}$, %it turns out that
	\begin{equation}\label{contract property}
		\begin{aligned}
			\mathbb{E}\Big(\widehat{X}_{k\tau+r}^\top K_{k\tau+r}\widehat{X}_{k\tau+r}\Big)\le\mathbb{E}\Big(\widehat{\xi}^\top K_r\widehat{\xi}\Big)e^{-\frac{k\tau}{\alpha}}.
		\end{aligned}
	\end{equation}
	Thus
	\begin{equation*}
		\begin{aligned}			\mathbb{E}\Big(\widehat{X}_{k\tau+r}^\top\widehat{X}_{k\tau+r}\Big)\le\frac{1}{\delta}\mathbb{E}\Big(\widehat{X}_{k\tau+r}^\top K_{k\tau+r}\widehat{X}_{k\tau+r}\Big)\le\frac{\alpha }{\delta}e^{-\frac{k\tau}{\alpha}}\mathbb{E}\big|\widehat{\xi}\big|^2.
		\end{aligned}
	\end{equation*}

On the other hand, by Proposition \ref{linear one order BSDE}, BSDE
	\begin{equation*}
		\begin{aligned}
			\d \eta_t=&-\big(A_t^\top\eta_t+C_t^\top\zeta_t+K_tb_t+C_t^\top K_t\sigma_t+L_t\sigma_t\big)\d t+\zeta_t\d W_t
		\end{aligned}
	\end{equation*}
admits a unique $\tau$-random periodic solution $(\eta,\zeta)\in S_{\mathbb{F}}^2(0,\infty;\mathbb{S}^n)\times L_{\mathbb{F}}^{2,loc}(0,\infty;\mathbb{S}^n)$, where $(K,L)$ is the $\tau$-random periodic solution to \eqref{linear BSDE Lambda=I}.	Applying It\^{o} formula to $\big\langle K_tX^{r,\xi}_t,X^{r,\xi}_t\big\rangle+2\big\langle\eta_t,X^{r,\xi}_t\big\rangle$, we have
\begin{equation*}
		\begin{aligned}
			&\d\mathbb{E}\Big(\big\langle K_tX^{r,\xi}_t,X^{r,\xi}_t\big\rangle+2\big\langle\eta_t,X^{r,\xi}_t\big\rangle\Big)\\
			%=&\mathbb{E}\Big(-\big\langle X^{r,\xi}_t,X^{r,\xi}_t\big\rangle+\big\langle K_t\sigma_t,\sigma_t\big\rangle+2\big\langle b_t,\eta_t\big\rangle+2\big\langle \sigma_t,\zeta_t\big\rangle\Big)\\
			=&\mathbb{E}\Big(-\frac{\alpha}{1+\alpha}\big\langle X^{r,\xi}_t,X^{r,\xi}_t\big\rangle-\frac{1}{1+\alpha}\big\langle X^{r,\xi}_t,X^{r,\xi}_t\big\rangle+\big\langle K_t\sigma_t,\sigma_t\big\rangle+2\big\langle b_t,\eta_t\big\rangle+2\big\langle \sigma_t,\zeta_t\big\rangle\Big)&\nonumber\\
			\le&\mathbb{E}\Big(-\frac{1}{1+\alpha}\big\langle K_tX^{r,\xi}_t,X^{r,\xi}_t\big\rangle-\frac{1}{1+\alpha}\Big(-\big\langle \eta_t,\eta_t\big\rangle+2\big\langle \eta_t,X^{r,\xi}_t\big\rangle\Big)+\big\langle K_t\sigma_t,\sigma_t\big\rangle\\
			&\ \ \ \ +2\big\langle b_t,\eta_t\big\rangle+2\big\langle \sigma_t,\zeta_t\big\rangle\Big)\\
			=&\mathbb{E}\Big(-\frac{1}{1+\alpha}\Big(\big\langle K_tX^{r,\xi}_t,X^{r,\xi}_t\big\rangle+2\big\langle \eta_t,X^{r,\xi}_t\big\rangle\Big)+\frac{1}{1+\alpha}\big\langle\eta_t,\eta_t\big\rangle+\big\langle K_t\sigma_t,\sigma_t\big\rangle\\
			&\ \ \ \ +2\big\langle b_t,\eta_t\big\rangle+2\big\langle\sigma_t,\zeta_t\big\rangle\Big).
		\end{aligned}
	\end{equation*}
	By Gronwall inequality, for any $r\in[0,\tau)$, $k\in\mathbb{N}$,
	\begin{equation*}
		\begin{aligned}
			&\mathbb{E}\Big(\frac{1}{2}\big\langle K_{k\tau+r}X^{r,\xi}_{k\tau+r},X^{r,\xi}_{k\tau+r}\big\rangle-2\big\langle K_{k\tau+r}^{-1}\eta_{k\tau+r},\eta_{k\tau+r}\big\rangle\Big)\\
			\le&\mathbb{E}\Big(\big\langle K_{k\tau+r}X^{r,\xi}_{k\tau+r},X^{r,\xi}_{k\tau+r}\big\rangle+2\big\langle \eta_{k\tau+r},X^{r,\xi}_{k\tau+r}\big\rangle\Big)\\
			\le&\mathbb{E}\Big(\big\langle K_r\xi,\xi\big\rangle+2\big\langle \eta_r,\xi\big\rangle\Big)e^{-\frac{k\tau}{1+\alpha}}\\
			&+\mathbb{E}\Big(\int_r^{k\tau+r} e^{-\frac{k\tau+r-s}{1+\alpha}}\Big(\frac{1}{1+\alpha}\big\langle \eta_s,\eta_s\big\rangle+\big\langle K_s\sigma_s,\sigma_s\big\rangle+2\big\langle b_s,\eta_s\big\rangle+2\big\langle \sigma_s,\zeta_s\big\rangle\Big)\d s\Big).
		\end{aligned}
	\end{equation*}
	Thus
\begin{equation*}
		\begin{aligned}
			&\mathbb{E}\big|X^{r,\xi}_{k\tau+r}\big|^2\\
			\le&\frac{2}{\delta}\Big(\frac{2}{\delta}\sup\limits_{s\in[k\tau,(k+1)\tau]}\mathbb{E}\big|\eta_s\big|^2+\Big((\alpha+1)\mathbb{E}|\xi|^2+\sup\limits_{s\in[0,\tau]}\mathbb{E}|\eta_s|^2\Big)e^{-\frac{k\tau}{1+\alpha}}\\
			&+\mathbb{E}e^{\frac{\tau}{1+\alpha}}\int_0^{(k+1)\tau}e^{-\frac{(k+1)\tau-s}{1+\alpha}}\Big(\frac{1}{1+\alpha}\big|\eta_s\big|^2+\alpha\big|\sigma_s\big|^2+\big|\eta_s\big|^2+\big|b_s\big|^2+\big|\zeta_s\big|^2+\big|\sigma_s\big|^2\Big)\d s\Big)\\
			\le&\Big(\frac{2}{\delta}\big(\frac{2}{\delta}+1+\alpha\big)\Big)\Big(\sup\limits_{s\in[0,\tau]}\mathbb{E}|\eta_s|^2+\mathbb{E}|\xi|^2\Big)\\
			&+\frac{2}{\delta}\big(\frac{1}{1+\alpha}+1+\alpha\big) e^{\frac{\tau}{1+\alpha}}\big(1-e^{-\frac{\tau}{1+\alpha}}\big)^{-1}\mathbb{E}\int_0^\tau\big(\big|\eta_s\big|^2+\big|\zeta_s\big|^2+\big|b_s\big|^2+\big|\sigma_s\big|^2\big)\d s,
		\end{aligned}
	\end{equation*}
	which puts the end of the proof.
\end{proof}

In particular, if the initial time of state equation is $0$, we immediately have the following corollary.
\begin{cor}\label{priori estimate zero intial time} For a given $\tau>0$, assume that
	$[A,C]$ is $\tau$-random periodic mean-square exponentially stable and $b,\sigma \in\mathscr{B}_\tau(\mathbb{R}^n)$. Then there exists $\beta,\lambda>0$ such that SDE with the $\mathbb{F}$-independent initial state $\xi\in L_{\mathscr{F}}^2(\Omega;\mathbb{R}^n)$
	\begin{equation*}\label{semi linear SDE zero intial time}
		\left\{\begin{aligned}
			d X_t&=\big(A_tX_t+b_t\big)\d t+\big(C_tX_t+\sigma_t\big)\d W_t,\ \ \ t\ge 0,\\
			X_0&=\xi.
		\end{aligned}\right.
	\end{equation*}
admits a unique solution $X^{0,\xi}\in L_{\mathbb{F}}^{2,loc}(0,\infty;\mathbb{R}^{n})$,	and for any $\mathbb{F}$-independent $\xi, \xi_1, \xi_2\in L_{\mathscr{F}_r}^2(\Omega;\mathbb{R}^n)$, $t\ge0$,
	\begin{equation}\label{delta square estimate zero intial time}
		\begin{aligned}
			\mathbb{E}\big|X^{0,\xi_1}_t-X^{0,\xi_2}_t\big|^2\le \beta e^{-\lambda t}\mathbb{E}|\xi_1-\xi_2|^2
		\end{aligned}
	\end{equation}
and
	\begin{equation}\label{square estimate zero intial time}
		\begin{aligned}
			\mathbb{E}\big|X^{0,\xi}_t\big|^2\le \beta(1+\mathbb{E}|\xi|^2).
		\end{aligned}
	\end{equation}
\end{cor}
\begin{proof} We only need to prove \eqref{delta square estimate zero intial time} and \eqref{square estimate zero intial time}. For any $t=k\tau+r$, $r\in[0,\tau)$, $k\in\mathbb{N}$, by \eqref{delta square estimate} and \eqref{square estimate},
	\begin{equation*}
		\begin{aligned}			\mathbb{E}\big|X^{0,\xi_1}_t-X^{0,\xi_2}_t\big|^2=&\mathbb{E}\big|X^{r,X^{0,\xi_1}_r}_{k\tau+r}-X^{r,X^{0,\xi_2}_r}_{k\tau+r}\big|^2\le c e^{-\lambda k\tau}\mathbb{E}|X^{0,\xi_1}_r-X^{0,\xi_2}_r|^2\\
			\le& ce^{\lambda \tau}e^{-\lambda (k\tau+r)}\sup\limits_{s\in[0,\tau]}\mathbb{E}|X^{0,\xi_1}_s-X^{0,\xi_2}_s|^2\le c e^{-\lambda t}\mathbb{E}|\xi_1-\xi_2|^2
	\end{aligned}
	\end{equation*}
and
	\begin{equation*}
		\begin{aligned}			
\mathbb{E}\big|X^{0,\xi}_t\big|^2=\mathbb{E}\big|X^{r,X^{0,\xi}_r}_{k\tau+r}\big|^2
			%\le&\beta(1+\mathbb{E}|X^{0,\xi}_r|^2)\\
			\le c(1+\sup\limits_{s\in[0,\tau]}\mathbb{E}|X^{0,\xi}_s|^2)\le c(1+\mathbb{E}|\xi|^2).
		\end{aligned}
	\end{equation*}
The desired estimates follows.
\end{proof}

Next we show how the shift operator of Brownian path $\theta$ shifts the solution to SDE \eqref{semi linear SDE}.
\begin{prop}\label{periodic time shift property} For a given $\tau>0$, assume
	$[A,C]$ is $\tau$-random periodic mean-square exponentially stable and $b,\sigma\in\mathscr{B}_\tau(\mathbb{R}^n)$. Then the solution $X^{0,\xi}$ to SDE \eqref{semi linear SDE} with the initial time $0$ and $\mathbb{F}$-independent initial state $\xi\in L_{\mathscr{F}}^2(\Omega;\mathbb{R}^n)$ satisfies for any $t\ge 0$, $k\in\mathbb{N}$,
	\begin{equation*}
		\begin{aligned}
			X^{k\tau,\theta_{k\tau}\circ\xi}_{t+k\tau}=\theta_{k\tau}\circ X^{0,\xi}_{t}\ \ {\text a.s.}
		\end{aligned}
	\end{equation*}
\end{prop}
\begin{proof} Notice that for any $t\ge 0$, $k\in\mathbb{N}$, $X^{k\tau,\theta_{k\tau}\circ\xi}_{t+k\tau}$ satisfies SDE
	\begin{equation*}
		\left\{\begin{aligned}
			\d X_t=&\big(A_{t+k\tau}X_t+b_{t+k\tau}\big)\d t+\big(C_{t+k\tau}X_t+\sigma_{t+k\tau}\big)\d W_{t+k\tau},\\
			X_0=&\theta_{k\tau}\circ\xi.\\	
		\end{aligned}\right.
	\end{equation*}
	On the other hand, for any $t\ge 0$, $k\in\mathbb{N}$, $\theta_{k\tau}\circ X^{0,\xi}_{t}$ satisfies SDE
	\begin{equation*}
		\left\{\begin{aligned}
			\d X_t=&\big((\theta_{k\tau}\circ A_{t})X_t+\theta_{k\tau}\circ b_{t}\big)\d t+\big((\theta_{k\tau}\circ C_{t})X_t+\theta_{k\tau}\circ\sigma_{t}\big)\d (\theta_{k\tau}\circ W_{t}),\\
			X_0=&\theta_{k\tau}\circ\xi.\\	
		\end{aligned}\right.
	\end{equation*}
	Bearing in mind that $\theta_{k\tau}\circ W_{t}=W_{t+k\tau}-W_{k\tau}$, $A,C,b,\sigma$ are $\tau$-random periodic %and are adapted to $\mathscr{F}_{k\tau,t+k\tau}^W$
and the measurability of $X^{k\tau,\theta_{k\tau}\circ\xi}_{t+k\tau}$ and $\theta_{k\tau}\circ X^{0,\xi}_{t}$ with respect to $\mathscr{F}_{k\tau,t+k\tau}^W\vee\sigma(\theta_{k\tau}\circ\xi)$, we know that the above two equations coincides. So Proposition \ref{periodic time shift property} follows from the uniqueness of strong solution to SDE \eqref{semi linear SDE} follows.
\end{proof}

Then we show that the solution to state equation could be $\tau$-random periodic if its initial state is properly chosen.
\begin{prop}\label{random periodic SDE} For a given $\tau>0$, assume that
	$[A,C]$ is $\tau$-random periodic mean-square exponentially stable and $b,\sigma\in\mathscr{B}_\tau(\mathbb{R}^n)$. Then there exists a unique $\mathbb{F}$-independent $\xi^*\in L_{\mathscr{F}}^2(\Omega;\mathbb{R}^n)$ such that $X^{0,\xi^*}$ is the solution to SDE \eqref{semi linear SDE} with the initial time $0$ and initial state $\xi^*$ satisfying for any $t\ge 0$,
	\begin{equation}\label{random periodic property}
		\begin{aligned}
			X^{0,\xi^*}_{t+\tau}=\theta_{\tau}\circ X^{0,\xi^*}_t\ \ {\text a.s.}
		\end{aligned}
	\end{equation}
\end{prop}
\begin{proof}Set
	\begin{equation*}
		\begin{aligned}
			\mathscr{G}=\Big\{\xi\in L_{\mathscr{F}}^2(\Omega;\mathbb{R}^n)\Big|\xi\text{ is independent of }\mathbb{F}\Big\}.
		\end{aligned}
	\end{equation*}
	For any $\xi\in\mathscr{G}$ and at time $\tau$, the solution $X^{0,\xi}_\tau$ to SDE \eqref{semi linear SDE} with the initial time $0$ and initial state $\xi$ is independent of $\mathscr{F}_{\tau,\infty}^W$, and thus $\theta_\tau^{-1}\circ X^{0,\xi}_\tau$ is independent of $\mathbb{F}$.\\
	By Proposition \ref{priori estimate} and measure preserving property of $\theta_\tau$
	\begin{equation*}
		\begin{aligned}
			\mathbb{E}\big|\theta_\tau^{-1}\circ X^{0,\xi}_\tau\big|^2=\mathbb{E}\big|X^{0,\xi}_\tau\big|^2\le c(1+\mathbb{E}|\xi|^2),
		\end{aligned}
	\end{equation*}
	which leads to
$\theta_\tau^{-1}\circ X^{0,\xi}_{\tau}\in L_{\mathscr{F}}^2(\Omega;\mathbb{R}^n)$.
	Then define a measurable mapping $\psi:\mathscr{G}\rightarrow \mathscr{G}$ as below
	\begin{equation*}
		\begin{aligned}
			\psi(\xi)=\theta_\tau^{-1}\circ X^{0,\xi}_\tau.
		\end{aligned}
	\end{equation*}
	Notice that $L_{\mathscr{F}}^2(\Omega;\mathbb{R}^n)$ is equipped with the common distance
	\begin{equation*}
		\begin{aligned}
			d(\xi_1,\xi_2)=\mathbb{E}\Big((\xi_1-\xi_2)^\top(\xi_1-\xi_2)\Big).
		\end{aligned}
	\end{equation*}
	Let $K$ be the solution to \eqref{linear BSDE Lambda=I}, and then
	\begin{equation*}
		\begin{aligned}
			d_0(\xi_1,\xi_2)=\mathbb{E}\Big((\xi_1-\xi_2)^\top K_0(\xi_1-\xi_2)\Big).
		\end{aligned}
	\end{equation*}
introduce a new distance in $L_{\mathscr{F}}^2(\Omega;\mathbb{R}^n)$ since $K_0\in\mathbb{S}_+^n$. Consequently, $d$ is equivalent to $d_0$, and $L_{\mathscr{F}}^2(\Omega;\mathbb{R}^n)$ equipped with $d_0$ is a Banach space. Noticing $\mathscr{G}\subset L_{\mathscr{F}}^2(\Omega;\mathbb{R}^n)$ is a closed linear space, we know that $\mathscr{G}$ equipped with $d_0$ is a Banach space.
Moreover, by \eqref{contract property}
	\begin{equation*}\label{contraction mapping}
		\begin{aligned}
			\mathbb{E}\Big(\widehat{X}_\tau^\top K_0\widehat{X}_\tau\Big)\le\mathbb{E}\Big(\widehat{\xi}^\top K_0\widehat{\xi}\Big)e^{-\lambda\tau}.
		\end{aligned}
	\end{equation*}
Hence $\psi$ is a contraction mapping under $d_0$. By the fixed-point theorem, there exists a unique $\xi^*\in\mathscr{G}$ such that
	\begin{equation*}
		\begin{aligned}
			\theta_\tau^{-1}\circ X^{0,\xi^*}_\tau=\xi^*\ \ {\text a.s.}
		\end{aligned}
	\end{equation*}
Therefore, by Proposition \ref{periodic time shift property}, we have
	\begin{equation*}
		\begin{aligned}
			\theta_\tau\circ X^{0,\xi^*}_t=X^{\tau,\theta_\tau\circ\xi^*}_{t+\tau}=X^{\tau,X^{0,\xi^*}_\tau}_{t+\tau}=X^{0,\xi^*}_{t+\tau}\ \ {\text a.s.}
		\end{aligned}
	\end{equation*}
\end{proof}

\section{Random periodic solutions to stochastic Riccati equations}
To begin with, we formulate our concerned control problem. For a $\mathbb{R}^m$-valued control process $u$ and a given $x\in\mathbb{R}^n$, the linear state equation has a form
\begin{equation}\label{wz3}
	\left\{\begin{aligned}
		\d X_t=&(A_tX_t+B_tu_t+b_t)\d t+(C_tX_t+\sigma_t)\d W_t,\ \ \ t\ge 0,\\
		X_0=&x,
	\end{aligned}\right.
\end{equation}
where $A,C:[0,\infty)\times\Omega\rightarrow\mathbb{R}^{n\times n}$, $B:[0,\infty)\times\Omega\rightarrow\mathbb{R}^{n\times m},b:[0,\infty)\times\Omega\rightarrow\mathbb{R}^n$ and $\sigma:[0,\infty)\times\Omega\rightarrow\mathbb{R}^n$ are $\mathbb{F}$-progressively measurable functions.

For $T>0$, define
\begin{equation}\label{wz4}
	\begin{aligned}
		J_T(x,u)=\mathbb{E}\int_0^TF(t,X_t,u_t)\d t,
	\end{aligned}
\end{equation}
where
\begin{equation*}
	\begin{aligned}
		F(t,X_t,u_t)
		=&\left\langle Q_tX_t,X_t\right\rangle+2\left\langle S_tX_t,u_t\right\rangle+\left\langle R_tu_t,u_t\right\rangle+2\left\langle q_t,X_t\right\rangle+2\left\langle \rho_t,u_t\right\rangle,
	\end{aligned}
\end{equation*}
and $Q:[0,\infty)\times\Omega\rightarrow\mathbb{S}^{n}$, $S:[0,\infty)\times\Omega\rightarrow\mathbb{R}^{m\times n}$, $R:[0,\infty)\times\Omega\rightarrow\mathbb{S}^{m}$, $q:[0,\infty)\times\Omega\rightarrow\mathbb{R}^n$, $\rho:[0,\infty)\times\Omega\rightarrow\mathbb{R}^m$ are $\mathbb{F}$-progressively measurable functions.

The cost functional is defined as
\begin{equation*}
	\begin{aligned}
		\mathcal{E}(x,u)=\varliminf\limits_{T \to \infty}\frac{1}{T}J_T(x,u).
	\end{aligned}
\end{equation*}

We study the closed-loop control for above state equation and cost functional. The control-state pair $(u^{x,\Theta,v}, X^{x,\Theta,v})$ satisfies
\begin{equation}\label{wz1}
	\begin{aligned}
		u^{x,\Theta,v}_t=\Theta_tX^{x,\Theta,v}_t+v_t,\ \ \ t\ge 0
	\end{aligned}
\end{equation}
and the closed-loop state equation becomes
\begin{equation}\label{wz2}
		\left\{\begin{aligned}
			\d X^{x,\Theta,v}_t=&\Big((A_t+B_t\Theta_t)X^{x,\Theta,v}_t+B_tv_t+b_t\Big)\d t+\Big(C_tX^{x,\Theta,v}_t+\sigma_t\Big)\d W_t,\ \ \ t\ge 0,\\
			X^{x,\Theta,v}_0=&x,
		\end{aligned}\right.
\end{equation}
where $(\Theta,v):[0,\infty)\times\Omega\rightarrow\mathbb{R}^{n\times n}\times\mathbb{R}^{n}$ are $\mathbb{F}$-progressively measurable functions.

Denote by $\mathbb{U}$ a set of measurable functions $(\Theta,v)$ which guarantees the control-state pair $(u^{x,\Theta,v}, X^{x,\Theta,v})\in L_{\mathbb{F}}^{2,loc}(0,\infty;\mathbb{R}^{m})\times L_{\mathbb{F}}^{2,loc}(0,\infty;\mathbb{R}^{n})$ and will be further specified. Then we define the admissible set
\begin{equation*}
	\begin{aligned}		
		\mathcal{U}=\Big\{u^{x,\Theta,v}\in L_{\mathbb{F}}^{2,loc}(0,\infty;\mathbb{R}^{m})\Big|u^{x,\Theta,v}\ {\text satisfies}\ (\ref{wz1})\ {\text with}\ (\Theta,v)\in\mathbb{U}\Big\}.
	\end{aligned}
\end{equation*}

Now the concerned ergodic LQ optimal control problem is given as below.

\textbf{Problem (CL-LQE).} Find a closed-loop control $u^{x,\bar{\Theta},\bar{v}}\in\mathcal{U}$ such that
\begin{equation*}
	\begin{aligned}
		\mathcal{E}(x,u^{x,\bar{\Theta},\bar{v}})=
		\inf\limits_{u^{x,\Theta,v}\in\mathcal{U}}\mathcal{E}(x,u^{x,\Theta,v}).
	\end{aligned}
\end{equation*}

\textbf{Problem (CL-LQE)} is called closed-loop solvable, if there exists an optimal control $u^{x,\bar{\Theta},\bar{v}}\in\mathcal{U}$, which together with its corresponding optimal state $X^{x,\bar{\Theta},\bar{v}}$ constitutes an optimal control-state pair. The function $V(x)=\inf\limits_{u^{x,\Theta,v}\in\mathcal{U}}\mathcal{E}(x,u^{x,\Theta,v})$ is called the value function of \textbf{Problem (CL-LQE)}.

To specify $\mathbb{U}$ in the admissible set, we introduce the mean-square exponentially stabilizable condition.
\begin{defn} For a given $\tau>0$, assume $A,C\in\mathscr{B}_\tau(\mathbb{R}^{n\times n})$ and $B,D\in\mathscr{B}_\tau(\mathbb{R}^{n\times m})$. The 4-tuple of coefficients $[A,C;B,D]$ is $\tau$-random periodic mean-square exponentially stabilizable if there exist $\Theta\in\mathscr{B}_\tau(\mathbb{R}^{m\times n})$ and $\beta,\lambda,\delta>0$ such that the homogeneous system
	\begin{equation*}
		\left\{\begin{aligned}
			\d\Phi_t&=(A_t+B_t\Theta_t)\Phi_t\d t+(C_t+D_t\Theta_t)\Phi_t\d W_t,\ \ \ t\ge 0,\\
			\Phi_0&=I_n,
		\end{aligned}\right.
	\end{equation*}
	has a unique solution $\Phi\in L_{\mathbb{F}}^{2}(0,\infty;\mathbb{R}^{n\times n})$ and for any $t\ge 0$,
	\begin{equation*}
		\begin{aligned}
			&\mathbb{E}\big|\Phi_t\big|^2\le\beta e^{-\lambda t},\\
		&\inf\limits_{r\in[0,\tau)}\mathbb{E}\Big(\int_r^\infty(\Phi_s\Phi_r^{-1})^\top\Phi_s\Phi_r^{-1}\d s\Big|\mathscr{F}_r\Big)\ge\delta I_n.
		\end{aligned}
	\end{equation*}
Here $\Theta$ is called a $\tau$-random periodic stabilizer of $[A,C;B,D]$
and the set of all $\tau$-periodic stabilizers of $[A,C;B,D]$ is denoted by $\mathscr{S}_\tau[A,C;B,D]$. Obviously, for any $\Theta\in\mathscr{S}_\tau[A,C;B,0]$, $[A+B\Theta,C]$ is $\tau$-random periodic mean-square exponentially stable.
\end{defn}

Then we give the assumptions on coefficients of state equation (\ref{wz3}) and cost functional (\ref{wz4}).

\begin{ass}
	{\textbf{(A1)}} For a given $\tau>0$, $A,B,C,D,b,Q,S,R,q,\rho\in\mathscr{B}_\tau$.\\
	{\textbf{(A2)}} $R,Q-S^\top R^{-1}S$ are uniformly positive definite.\\
	{\textbf{(A3)}}  For a given $\tau>0$, $[A,C;B,0]$ is $\tau$-random periodic mean-square exponentially stabilizable.
\end{ass}

Define
\begin{equation*}
	\begin{aligned}				\mathbb{U}=\Big\{(\Theta,v)\Big|\Theta\in\mathscr{S}_\tau[A,C;B,0],v\in\mathscr{B}_\tau(\mathbb{R}^{m})\Big\}.
	\end{aligned}
\end{equation*}
Due to \textbf{(A1)} and \textbf{(A3)}, the above $\mathbb{U}$ guarantees $(u^{x,\Theta,v}, X^{x,\Theta,v})\in L_{\mathbb{F}}^{2,loc}(0,\infty;\mathbb{R}^{m})\times L_{\mathbb{F}}^{2,loc}(0,\infty;\mathbb{R}^{n})$ and the well-posedness of ergodic cost functional, i.e. for any $T>0$,
\begin{equation*}
	\begin{aligned}		
		\frac{1}{T}\mathbb{E}\int_0^TF(t,X^{x,\Theta,v}_t,u^{x,\Theta,v}_t)\d t<\infty.
	\end{aligned}
\end{equation*}

We then establish the equivalence between random periodic mean-square exponentially stabilizability and well-posedness of a BSDE.
\begin{prop}\label{Riccati type posedness} For a given $\tau>0$, assume $A,C\in\mathscr{B}_\tau(\mathbb{R}^{n\times n})$, $B\in\mathscr{B}_\tau(\mathbb{R}^{n\times m})$ and that $Q\in \mathscr{B}_\tau(\mathbb{S}_+^n)$, $R\in\mathscr{B}_\tau(\mathbb{S}_+^m)$ are uniformly positive definite. Then $[A,C;B,0]$ is $\tau$-random periodic mean-square exponentially stabilizable if and only if BSDE
	\begin{equation}\label{Riccati type equation}
		\begin{aligned}
			\d K_t=&-\big(K_tA_t+A_t^\top K_t+C_t^\top K_tC_t+L_tC_t+C_t^\top L_t+Q_t-K_tB_tR_t^{-1}B_t^\top K_t\big)\d t\\
			&+L_t\d W_t
		\end{aligned}
	\end{equation}
admits a unique $\tau$-random periodic solution $(K,L)\in L_{\mathbb{F}}^\infty(0,\infty;\mathbb{S}^n)\times L_{\mathbb{F}}^{2,loc}(0,\infty;\mathbb{S}^n)$ and $K\in\mathscr{B}_\tau(\mathbb{S}_+^n)$ is uniformly positive definite, and moreover,
\begin{equation*}
		\begin{aligned}
			-R^{-1}B^\top K\in\mathscr{S}_\tau[A,C;B,0].
		\end{aligned}
	\end{equation*}
\end{prop}
\begin{proof} Sufficiency. Suppose that \eqref{Riccati type equation} admits a unique $\tau$-random periodic solution $(K,L)\in L_{\mathbb{F}}^\infty(0,\infty;\mathbb{S}^n)\times L_{\mathbb{F}}^{2,loc}(0,\infty;\mathbb{S}^n)$ and $K\in\mathscr{B}_\tau(\mathbb{S}_+^n)$ is uniformly positive definite. Set
		\begin{equation*}
			\begin{aligned}
				\Theta_t\triangleq-R_t^{-1}B_t^\top K_t,\ \ \ t\geq0.
			\end{aligned}
		\end{equation*}
		Then \eqref{Riccati type equation} is rewritten to
		\begin{equation}\label{wz30}
			\begin{aligned}
				\d K_t=&-\Big(K_t(A_t+B_t\Theta_t)+(A_t+B_t\Theta_t)^\top K_t+C_t^\top K_tC_t+L_tC_t+C_t^\top L_t+Q_t\\
				&+K_tB_tR_t^{-1}B_t^\top K_t\Big)\d t+L_t\d W_t.
			\end{aligned}
		\end{equation}
		By Proposition \ref{stable condition characterization}, $[A+B\Theta,C]$ is $\tau$-random periodic mean-square exponentially stable, and hence $[A,C;B,0]$ is $\tau$-random periodic mean-square exponentially stabilizable.
		
		Necessity. Suppose that $[A,C;B,0]$ is $\tau$-random periodic mean-square exponentially stabilizable.
		
		We first prove the existence of solution to \eqref{Riccati type equation}. Let $\Theta^0\in\mathscr{S}_\tau[A,C;B,0]$. By Proposition \ref{linear BSDE}, BSDE
		\begin{equation}
			\begin{aligned}\label{equation 1}
				\d K^1_t=&-\Big(K^1_t(A_t+B_t\Theta^0_t)+(A_t+B_t\Theta^0_t)^\top K^1_t+C_t^\top K^1_tC_t+L^1_tC_t+C_t^\top L^1_t\\
				&+Q_t+(\Theta^0_t)^\top R_t\Theta^0_t\Big)\d t+L^1_t\d W_t
			\end{aligned}
		\end{equation}
		has a unique $\tau$-random periodic solution $(K^1,L^1)\in L_{\mathbb{F}}^\infty(0,\infty;\mathbb{S}^n)\times L_{\mathbb{F}}^{2,loc}(0,\infty;\mathbb{S}^n)$, and $K^1\in\mathscr{B}_\tau(\mathbb{S}_+^n)$ is uniformly positive definite.	Define
		\begin{equation*}
			\begin{aligned}
				\Theta^1_t=-R_t^{-1}B_t^\top K^1_t,\ \ \ t\geq0.
			\end{aligned}
		\end{equation*}
		Then \eqref{equation 1} becomes
		\begin{equation*}
			\begin{aligned}
				\d K^1_t=&-\Big(K^1_t(A_t+B_t\Theta^1_t)+(A_t+B_t\Theta^1_t)^\top K^1_t+C_t^\top K^1_tC_t+L^1_tC_t+C_t^\top L^1_t+Q_t\\
				&\ \ \ \ \ +(\Theta^1_t)^\top R_t\Theta^1_t+(\Theta^0_t-\Theta^1_t)^\top R_t(\Theta^0_t-\Theta^1_t)\Big)\d t+L^1_t\d W_t.
			\end{aligned}
		\end{equation*}
		By Proposition \ref{stable condition characterization}, $\Theta^1\in\mathscr{S}_\tau[A,C;B,0]$. For any $N\in\mathbb{N}^+$, define
		\begin{equation*}
			\begin{aligned}
				\Theta^N_t=-R_t^{-1}B_t^\top K^N_t,\ \ \ t\geq0,
			\end{aligned}
		\end{equation*}
		and a sequence of BSDEs
		\begin{equation}\label{iterative equation}
			\begin{aligned}
				\d K^{N+1}_t=&-\Big(K^{N+1}_t(A_t+B_t\Theta^N_t)+(A_t+B_t\Theta^N_t)^\top K^{N+1}_t+C_t^\top K^{N+1}_tC_t\\
				&\ \ \ \ \ +L^{N+1}_tC_t+C_t^\top L^{N+1}_t+Q_t+(\Theta^N_t)^\top R_t\Theta^N_t\Big)\d t+L^{N+1}_t\d W_t.
			\end{aligned}
		\end{equation}
		By induction, for any $N\in\mathbb{N}^+$, $\Theta^N\in\mathscr{S}_\tau[A,C;B,0]$, and \eqref{iterative equation} has a unique $\tau$-random periodic solution $(K^{N+1},L^{N+1})\in L_{\mathbb{F}}^\infty(0,\infty;\mathbb{S}^n)\times L_{\mathbb{F}}^{2,loc}(0,\infty;\mathbb{S}^n)$, and $K^{N+1}\in\mathscr{B}_\tau(\mathbb{S}_+^n)$ is uniformly positive definite.
		
		Note that for any $N\in\mathbb{N}^+$,
		\begin{equation*}
			\begin{aligned}
				&\d (K^N_t-K^{N+1}_t)\\
				=&-\Big((K^N_t-K^{N+1}_t)(A_t+B_t\Theta^N_t)+(A_t+B_t\Theta^N_t)^\top(K^N_t-K^{N+1}_t)\\
				&\ \ \ \ \ +C_t^\top (K^N_t-K^{N+1}_t)C_t+(L^N_t-L^{N+1}_t)C_t+C_t^\top (L^N_t-L^{N+1}_t)\\
				&\ \ \ \ \ +(\Theta^{N-1}_t-\Theta^N_t)^\top R_t(\Theta^{N-1}_t-\Theta^N_t)\Big)\d t+(L^N_t-L^{N+1}_t)\d W_t.
			\end{aligned}
		\end{equation*}
		It follows from Proposition \ref{linear BSDE} that, for any $t\ge0$,
		$K_t^N-K_t^{N+1}\ge 0$ a.s. By the monotone convergence theorem, there exists $K\in L_{\mathbb{F}}^\infty(0,\tau;\mathbb{S}^n)$ such that for any $0\leq t\leq\tau$, $K_t^N\rightarrow K_t$ as $N\rightarrow\infty$.
		On the other hand, for any $N_1,N_2\in\mathbb{N}^+$ and $N_1<N_2$,
		applying It\^{o} formula to $\big|K^{N_2+1}_t-K^{N_1+1}_t\big|^2$, we have
	\begin{equation*}
				\begin{aligned}			&\mathbb{E}\big|K^{N_2+1}_0-K^{N_1+1}_0\big|^2+\mathbb{E}\int_0^\tau\big|L^{N_2+1}_s-L^{N_1+1}_s\big|^2\d s\\
					=&\mathbb{E}\big|K^{N_2+1}_{\tau}-K^{N_1+1}_{\tau}\big|^2+2\mathbb{E}\int_0^\tau\text{tr}\Big((K^{N_2+1}_s-K^{N_1+1}_s)\Big((K^{N_2+1}_s-K^{N_1+1}_s)(A_s+B_s\Theta^{N_1}_s)\\
					&\ \ \ \ \ \ \ \ \ \ \ \ \ \ \ \ \ \ \ \   +(A_s+B_s\Theta^{N_1}_s)^\top(K^{N_2+1}_s-K^{N_1+1}_s)+K^{N_2+1}_sB_s(\Theta^{N_2}_s-\Theta^{N_1}_s)\\
					&\ \ \ \ \ \ \ \ \ \ \ \ \ \ \ \ \ \ \ \ +(\Theta^{N_2}_s-\Theta^{N_1}_s)^\top B_s^\top K^{N_2+1}_s+C_s^\top (K^{N_2+1}_s-K^{N_1+1}_s)C_s+(L^{N_2+1}_s-L^{N_1+1}_s)C_s\\
					&\ \ \ \ \ \ \ \ \ \ \ \ \ \ \ \ \ \ \ \ +C_s^\top (L^{N_2+1}_s-L^{N_1+1}_s)+(\Theta^{N_1}_s)^\top R_s(\Theta^{N_2}_s-\Theta^{N_1}_s)+(\Theta^{N_2}_s-\Theta^{N_1}_s)^\top R_s\Theta^{N_1}_s\\
					&\ \ \ \ \ \ \ \ \ \ \ \ \ \ \ \ \ \ \ \ +(\Theta^{N_2}_s-\Theta^{N_1}_s)^\top R_s(\Theta^{N_2}_s-\Theta^{N_1}_s)\Big)\Big)\d s\\
					\le&\mathbb{E}\big|K^{N_2+1}_{\tau}-K^{N_1+1}_{\tau}\big|^2+\frac{1}{2}\mathbb{E}\int_0^\tau\big|L^{N_2+1}_s-L^{N_1+1}_s\big|^2\d s\\
					&+c\mathbb{E}\int_0^\tau\big(\big|K^{N_2+1}_s-K^{N_1+1}_s\big|^2+\big|\Theta^{N_2}_s-\Theta^{N_1}_s\big|^2\big)\d s.
				\end{aligned}
		\end{equation*}
		Note that for any $0\leq t\leq\tau$, $\Theta^N_t\rightarrow-R_t^{-1}B_t^\top K_t$ as $N\rightarrow\infty$.
		By the dominated convergence theorem, for any $t\in[0,\tau]$, as $N_1,N_2\rightarrow\infty$,
		\begin{equation}\label{convergence of K}
			\begin{aligned}
				\mathbb{E}\big|K_t^{N_2+1}-K_t^{N_1+1}\big|^2\rightarrow 0\ \ {\text and}\ \ \mathbb{E}\int_0^\tau\big(\big|K^{N_2+1}_s-K^{N_1+1}_s\big|^2+\big|\Theta^{N_2}_s-\Theta^{N_1}_s\big|^2\big)\d s\rightarrow 0.
			\end{aligned}
		\end{equation}
		Hence as $N_1,N_2\rightarrow\infty$, we get from \eqref{convergence of K} that
		$\mathbb{E}\int_0^\tau\big|L^{N_2+1}_s-L^{N_1+1}_s\big|^2\d s\rightarrow 0$.
		This yields that $\big\{L^N\big\}_{N\in\mathbb{N}^+}$ is a Cauchy sequence in $L_{\mathbb{F}}^2(0,\tau;\mathbb{S}^n)$, and thus there exists an $L\in L_{\mathbb{F}}^2(0,\tau;\mathbb{S}^n)$ such that as $N\rightarrow\infty$,
		$\mathbb{E}\int_0^\tau\big|L^{N}_s-L_s\big|^2\d s\rightarrow 0$.

		Then we extend the definition of $(K,L)$ to $L_{\mathbb{F}}^\infty(0,\infty;\mathbb{S}^n)\times L_{\mathbb{F}}^{2,loc}(0,\infty;\mathbb{S}^n)$. For any $k\in\mathbb{N}^+$, define
		\begin{equation*}
			\begin{aligned}
				K_t=\theta_{k\tau}\circ K_{t-k\tau},\ \ \ L_t=\theta_{k\tau}\circ L_{t-k\tau},\ \ \ t\in(k\tau,(k+1)\tau],
			\end{aligned}
		\end{equation*}
		and as $N\rightarrow\infty$,
		\begin{equation*}
			\begin{aligned}
				&K^N_t=\theta_{k\tau}\circ K^N_{t-k\tau}\rightarrow\theta_{k\tau}\circ K_{t-k\tau}
			\end{aligned}
		\end{equation*}
		and
		\begin{equation*}
			\begin{aligned}
				&\mathbb{E}\int_{k\tau}^{(k+1)\tau}\big|L^N_s-\theta_{k\tau}\circ L_{s-k\tau}\big|^2\d s=\mathbb{E}\int_0^\tau\big|\theta_{k\tau}\circ L^N_s-\theta_{k\tau}\circ L_s\big|^2\d s\\
				=&\mathbb{E}\int_0^\tau\big|L^N_s-L_s\big|^2\d s\rightarrow 0.
			\end{aligned}
		\end{equation*}
		The above convergence together with the $\tau$-random periodicity of $A,B,C,Q,R$ implies that $(K,L)\in L_{\mathbb{F}}^\infty(0,\infty;\mathbb{S}^n)\times L_{\mathbb{F}}^{2,loc}(0,\infty;\mathbb{S}^n)$ is $\tau$-random periodic solution of \eqref{Riccati type equation}.
		
		As for the uniqueness of solution to \eqref{Riccati type equation}. Let $(K,L), (K',L')\in L_{\mathbb{F}}^\infty(0,\infty;\mathbb{S}^n)\times L_{\mathbb{F}}^{2,loc}(0,\infty;\mathbb{S}^n)$ be two $\tau$-random periodic solutions of \eqref{Riccati type equation}. Set
		\begin{equation*}
			\begin{aligned}
				\Theta_t\triangleq-R_t^{-1}B_t^\top K_t\ \ {\text and}\ \ \Theta'_t\triangleq-R_t^{-1}B_t^\top K'_t,\ \ \ t\geq0.
			\end{aligned}
		\end{equation*}
		Then $(K-K',L-L')$ satisfies BSDE
		\begin{equation*}\label{wz29}
			\begin{aligned}
				&\d (K_t-K'_t)\\
				=&-\Big((K_t-K'_t)(A_t+B_t\Theta_t)+(A_t+B_t\Theta_t)^\top(K_t-K'_t)+C_t^\top (K_t-K'_t)C_t\\
				&+(L_t-L'_t)C_t+C_t^\top (L_t-L'_t)+(\Theta_t-\Theta'_t)^\top R_t(\Theta_t-\Theta'_t)\Big)\d t+(L_t-L'_t)\d W_t.
			\end{aligned}
		\end{equation*}
		By Proposition \ref{linear BSDE}, for any $t\ge0$, $K_t\ge K'_t$ a.s.
		and similarly, $K'_t\ge K_t$ a.s. Thus for any $t\ge0$, $K_t=K'_t$ a.s.
Applying It\^{o} formula to $\big|K_t-K'_t\big|^2$, we have
		\begin{equation*}
			\begin{aligned}
				\mathbb{E}\big|K_0-K'_0\big|^2+\mathbb{E}\int_0^\tau\big|L_s-L'_s\big|^2\d s
				\le\mathbb{E}\big|K_\tau-K'_\tau\big|^2+\frac{1}{2}\mathbb{E}\int_0^\tau\big|L_s-L'_s\big|^2\d s+c \mathbb{E}\int_0^\tau\big|K_s-K'_s\big|^2\d s.
			\end{aligned}
		\end{equation*}
		Then for any $t\in[0,\tau]$, $L_t=L'_t$ a.s.
Due to the $\tau$-random periodicity of $L$, it yields that for any $t\ge0$,
$L_t=L'_t$ a.s.
		%Then the existence and uniqueness of the solution to \eqref{Riccati type equation} follows.
		
Finally,	since $Q\in \mathscr{B}_\tau(\mathbb{S}_+^n)$ is uniformly positive definite and  $\Theta^N_t\rightarrow-R_t^{-1}B_t^\top K_t$ as $N\rightarrow\infty$, there exists a $\widehat{N}\in\mathbb{N}$ such that $Q-(\Theta^{\widehat{N}}+R^{-1}B^\top K)^\top R(\Theta^{\widehat{N}}+R^{-1}B^\top K)\in \mathscr{B}_\tau(\mathbb{S}_+^n)$ is uniformly positive definite.
We rewrite \eqref{Riccati type equation} to
		\begin{equation*}
			\begin{aligned}
				\d K_t=&-\Big(K_t(A_t+B_t\Theta^{\widehat{N}}_t)+(A_t+B_t\Theta^{\widehat{N}}_t)^\top K_t+C_t^\top K_tC_t+L_tC_t+C_t^\top L_t+Q_t\\
				&+(\Theta^{\widehat{N}}_t)^\top R_t\Theta^{\widehat{N}}_t-(\Theta^{\widehat{N}}_t+R_t^{-1}B_t^\top K_t)^\top R_t(\Theta^{\widehat{N}}_t+R_t^{-1}B_t^\top K_t)\Big)\d t+L_t\d W_t.
			\end{aligned}
		\end{equation*}
By Proposition \ref{linear BSDE}, 	$K\in \mathscr{B}_\tau(\mathbb{S}_+^n)$ is uniformly positive definite. Moreover, noticing \eqref{wz30} which also comes from \eqref{Riccati type equation}, we further know $-R^{-1}B^\top K\in\mathscr{S}_\tau[A,C;B,0]$ by Proposition \ref{stable condition characterization}.
	\end{proof}

Based on Proposition \ref{Riccati type posedness}, the equivalence between random periodic mean-square exponentially stabilizability and well-posedness of the stochastic Riccati equation in \textbf{Problem (CL-LQE)} can also be obtained.
\begin{cor}\label{Riccati posedness} For a given $\tau>0$, assume $A,C\in\mathscr{B}_\tau(\mathbb{R}^{n\times n})$, $B\in\mathscr{B}_\tau(\mathbb{R}^{n\times m})$, $Q\in \mathscr{B}_\tau(\mathbb{S}^n)$, $R\in\mathscr{B}_\tau(\mathbb{S}^m)$, $S\in \mathscr{B}_\tau(\mathbb{R}^{m\times n})$ and that $Q-S^\top R^{-1}S\in\mathscr{B}_\tau(\mathbb{S}_+^n)$, $R\in\mathscr{B}_\tau(\mathbb{S}_+^m)$ are uniformly positive definite. Then $[A,C;B,0]$ is $\tau$-random periodic mean-square exponentially stabilizable if and only if the stochastic Riccati equation
	\begin{equation}\label{Riccati equation}
		\begin{aligned}
			\d K_t=&-\Big(K_tA_t+A_t^\top K_t+C_t^\top K_tC_t+L_tC_t+C_t^\top L_t+Q_t-(B_t^\top K_t+S_t)^\top R_t^{-1}(B_t^\top K_t+S_t)\Big)\d t\\
&+L_t\d W_t
		\end{aligned}
	\end{equation}
admits a unique $\tau$-random periodic solution $(K,L)\in L_{\mathbb{F}}^\infty(0,\infty;\mathbb{S}^n)\times L_{\mathbb{F}}^{2,loc}(0,\infty;\mathbb{S}^n)$ and $K\in\mathscr{B}_\tau(\mathbb{S}_+^n)$ is uniformly positive definite, and moreover,
	\begin{equation*}
		\begin{aligned}
			-R^{-1}(B^\top K+S)\in\mathscr{S}_\tau[A,C;B,0].
		\end{aligned}
	\end{equation*}
\end{cor}
\begin{proof} Set
	\begin{equation*}
		\begin{aligned}
			\widetilde{A}_t\triangleq A_t-B_tR^{-1}_tS_t,\ \ \ \widetilde{Q}_t\triangleq Q_t-S_t^\top R^{-1}_tS_t,\ \ \ t\geq0.
		\end{aligned}
	\end{equation*}
Then \eqref{Riccati equation} can be rewritten as
	\begin{equation*}
		\begin{aligned}
			\d K_t=-\big(K_t\widetilde{A}_t+\widetilde{A}_t^\top K_t+C_t^\top K_tC_t+L_tC_t+C_t^\top L_t+\widetilde{Q}_t-K_tB_tR_t^{-1}B_t^\top K_t\big)\d t+L_t\d W_t.
		\end{aligned}
	\end{equation*}
Note that $[A,C;B,0]$ is $\tau$-random periodic mean-square exponentially stabilizable is equivalent to $[\widetilde{A},C;B,0]$ is $\tau$-random periodic mean-square exponentially stabilizable, which together with Proposition \ref{Riccati type posedness} proves the desired result.
\end{proof}

\section{The explicit optimal controls}

Revisit the closed-loop state equation \eqref{wz2}, and according to Proposition \ref{random periodic SDE} we immediately have the following result for state equation.
\begin{pro}\label{wz31} For a given $\tau>0$, assume $A,C\in\mathscr{B}_\tau(\mathbb{R}^{n\times n})$, $B\in\mathscr{B}_\tau(\mathbb{R}^{n\times m})$ and that $[A,C;B,0]$ is $\tau$-random periodic mean-square exponentially stabilizable. \\ Then for any $(\Theta,v)\in\mathbb{U}$, there exists a unique $\mathbb{F}$-independent $\xi^{\Theta,v}\in L_{\mathscr{F}}^2(\Omega;\mathbb{R}^n)$ such that $X^{\xi^{\Theta,v},\Theta,v}$ is the solution to the state equation \eqref{wz2} with the initial time $0$ and initial state $\xi^{\Theta,v}$ satisfying for any $t\ge 0$,
\begin{equation*}\label{random periodic property closed loop}
	\begin{aligned}
		X^{\xi^{\Theta,v},\Theta,v}_{t+\tau}=\theta_{\tau}\circ X^{\xi^{\Theta,v},\Theta,v}_t\ \ \ {\text a.s.}
	\end{aligned}
\end{equation*}
\end{pro}

Then we prove that the cost functional $\mathcal{E}(x,u)$ on infinite horizon is equivalent to a cost functional on finite horizon.
\begin{thm}\label{transffered form by random periodic solution} Assume {\textbf{(A1)}}, {\textbf{(A3)}} and that $X^{x,\Theta,v}$ and $X^{\xi^{\Theta,v},\Theta,v}$, $(\Theta,v)\in\mathbb{U}$, are solutions to the state equations \eqref{wz2} with initial state $x\in\mathbb{R}^n$ and $\xi^{\Theta,v}\in L_{\mathscr{F}}^2(\Omega;\mathbb{R}^n)$, respectively, where $\xi^{\Theta,v}$ is given in Proposition \ref{wz31}. Then
	\begin{equation*}
		\begin{aligned}
			\mathcal{E}(x,u^{x,\Theta,v})=\varliminf\limits_{T \to \infty}\frac{1}{T}\mathbb{E}\int_0^TF(t,X^{x,\Theta,v}_t,u^{x,\Theta,v}_t)\d t
			=\frac{1}{\tau}\mathbb{E}\int_0^\tau F(t,X^{\xi^{\Theta,v},\Theta,v}_t,u^{\xi^{\Theta,v},\Theta,v}_t)\d t,
		\end{aligned}
	\end{equation*}
where $u^{x,\Theta,v}_t=\Theta_tX^{x,\Theta,v}_t
	+v_t$ and $u_t^{\xi^{\Theta,v},\Theta,v}=\Theta_tX^{\xi^{\Theta,v},\Theta,v}_t
	+v_t$.
\end{thm}
\begin{proof}
	For any $T>0$, there exists $N\in\mathbb{N}$ and $l\in[0,\tau)$ such that $T=N\tau+l$. Hence
	\begin{equation}\label{wz17}
		\begin{aligned}		
			&\frac{1}{T}\mathbb{E}\int_0^TF(t,X^{x,\Theta,v}_t,u^{x,\Theta,v}_t)\d t\\				=&\left(\frac{1}{T}-\frac{1}{N\tau}\right)\mathbb{E}\int_0^TF(t,X^{x,\Theta,v}_t,u^{x,\Theta,v}_t)\d t+\frac{1}{N\tau}\mathbb{E}\int_{N\tau}^TF(t,X^{x,\Theta,v}_t,u^{x,\Theta,v}_t)\d t\\		&+\frac{1}{N\tau}\mathbb{E}\sum_{k=0}^{N-1}\int_{k\tau}^{(k+1)\tau}F(t,X^{x,\Theta,v}_t,u^{x,\Theta,v}_t)\d t\\
			=&-\frac{l}{N\tau T}\mathbb{E}\int_0^TF(t,X^{x,\Theta,v}_t,u^{x,\Theta,v}_t)\d t+\frac{1}{N\tau}\mathbb{E}\int_0^lF(t+N\tau,X^{x,\Theta,v}_{t+N\tau},u^{x,\Theta,v}_{t+N\tau})\d t\\	&+\frac{1}{N\tau}\mathbb{E}\sum_{k=0}^{N-1}\int_{0}^{\tau}F(t+k\tau,X^{x,\Theta,v}_{t+k\tau},u^{x,\Theta,v}_{t+k\tau})\d t.
		\end{aligned}
	\end{equation}
	Since
	\begin{equation*}
		\begin{aligned}
			&F(t,X^{x,\Theta,v}_t,u^{x,\Theta,v}_t)\\
=&\left\langle(Q_t+S_t^\top\Theta_t+\Theta_t^\top S_t+\Theta_t^\top R_t\Theta_t)X^{x,\Theta,v}_t,X^{x,\Theta,v}_t\right\rangle\\
			&+2\left\langle(S_t+R_t\Theta_t)X^{x,\Theta,v}_t,v_t\right\rangle+2\left\langle X^{x,\Theta,v}_t,q_t+\Theta_t^\top\rho_t\right\rangle+\Big\langle R_tv_t,v_t\Big\rangle+2\Big\langle\rho_t,v_t\Big\rangle,
		\end{aligned}
	\end{equation*}
	there exist $H^a\in\mathscr{B}_\tau(\mathbb{S}^n)$, $H^b\in\mathscr{B}_\tau(\mathbb{R}^n)$ and $H^c\in\mathscr{B}_\tau(\mathbb{R})$ such that for any $t\ge 0$,
	\begin{equation*}
		\begin{aligned}	
			F(t,X^{x,\Theta,v}_t,u^{x,\Theta,v}_t)=\Big\langle H^a_tX^{x,\Theta,v}_t,X^{x,\Theta,v}_t\Big\rangle+\Big\langle H^b_t,X^{x,\Theta,v}_t\Big\rangle+H^c_t.
		\end{aligned}
	\end{equation*}
For the first two terms on the right hand of \eqref{wz17}, as $N\rightarrow\infty$ which is equivalent to $T\rightarrow\infty$, it follows from \eqref{square estimate zero intial time} that
	\begin{equation*}
		\begin{aligned}
			&\Big|\frac{l}{N\tau T}\mathbb{E}\int_0^TF(t,X^{x,\Theta,v}_t,u^{x,\Theta,v}_t)\d t\Big|+\Big|\frac{1}{N\tau}\mathbb{E}\int_0^lF(t+N\tau,X^{x,\Theta,v}_{t+N\tau},u^{x,\Theta,v}_{t+N\tau})\d t\Big|\\
			\le&\frac{lc}{N\tau}\sup\limits_{t\ge 0}\Big(1+\mathbb{E}|X^{x,\Theta,v}_t|^2\Big)\le\frac{c}{N}(1+|x|^2)\to0.
		\end{aligned}
	\end{equation*}

To deal with the last term on the right hand of \eqref{wz17}, take $(K,L)\in L_{\mathbb{F}}^\infty(0,\infty;\mathbb{S}^n)\times L_{\mathbb{F}}^{2,loc}(0,\infty;\mathbb{S}^n)$ to be the unique $\tau$-random periodic solution to  \eqref{linear BSDE Lambda=I} and set 	
\begin{equation*}
		\begin{aligned}
			\widehat{X}^{\Theta,v}_t\triangleq X_t^{x,\Theta,v}-X_t^{\xi^{\Theta,v},\Theta,v}.
		\end{aligned}
	\end{equation*}
By \eqref{delta square estimate zero intial time} and \eqref{square estimate zero intial time}, for any $t\in[0,\tau)$, we have
	\begin{equation*}
		\begin{aligned}
			&\Big|\mathbb{E}\Big(\Big\langle H^a_{t+k\tau}X^{x,\Theta,v}_{t+k\tau},X^{x,\Theta,v}_{t+k\tau}\Big\rangle-\Big\langle H^a_{t+k\tau}X^{\xi^{\Theta,v},\Theta,v}_{t+k\tau},X^{\xi^{\Theta,v},\Theta,v}_{t+k\tau}\Big\rangle\Big)\Big|\\
			%=&\Big|\mathbb{E}\Big(\Big\langle H^a_{t+k\tau}\widehat{X}^{\Theta,v}_{t+k\tau},\widehat{X}^{\Theta,v}_{t+k\tau}\Big\rangle+2\Big\langle H^a_{t+k\tau}X^{\xi^{\Theta,v},\Theta,v}_{t+k\tau},\widehat{X}^{\Theta,v}_{t+k\tau}\Big\rangle\Big)\Big|\\
			\le&\mathbb{E}\Big\langle H^a_{t+k\tau}K_{t+k\tau}^{-1}K_{t+k\tau}\widehat{X}^{\Theta,v}_{t+k\tau},\widehat{X}^{\Theta,v}_{t+k\tau}\Big\rangle+2\Big|\Big\langle K_{t+k\tau}^{-\frac{1}{2}}H^a_{t+k\tau}X^{\xi^{\Theta,v},\Theta,v}_{t+k\tau},K_{t+k\tau}^{\frac{1}{2}}\widehat{X}^{\Theta,v}_{t+k\tau}\Big\rangle\Big|\\
			\le&c\mathbb{E}\Big\langle K_{t+k\tau}\widehat{X}^{\Theta,v}_{t+k\tau},\widehat{X}^{\Theta,v}_{t+k\tau}\Big\rangle+c\Big(\mathbb{E}\big|X^{\xi^{\Theta,v},\Theta,v}_{t+k\tau}\big|^2\Big)^{\frac{1}{2}}\Big(\mathbb{E}\Big\langle K_{t+k\tau}\widehat{X}^{\Theta,v}_{t+k\tau},\widehat{X}^{\Theta,v}_{t+k\tau}\Big\rangle\Big)^{\frac{1}{2}}\\
			%\le&\beta\mathbb{E}\Big\langle K_t\widehat{X}^{\Theta,v}_t,\widehat{X}^{\Theta,v}_t\Big\rangle e^{-\lambda k\tau}+\beta(1+\mathbb{E}|\xi^{\Theta,v}|^2)^{\frac{1}{2}}\Big(\mathbb{E}\Big\langle K_t\widehat{X}^{\Theta,v}_t,\widehat{X}^{\Theta,v}_t\Big\rangle e^{-\lambda k\tau}\Big)^{\frac{1}{2}}\\
			\le&c(1+\mathbb{E}|x-\xi^{\Theta,v}|^2)e^{-\lambda k\tau}+c(1+\mathbb{E}|\xi^{\Theta,v}|^2)^{\frac{1}{2}}(1+\mathbb{E}|x-\xi^{\Theta,v}|^2)^{\frac{1}{2}}e^{-\frac{\lambda}{2} k\tau}\\
			\le&c(1+|x|^2+\mathbb{E}|\xi^{\Theta,v}|^2)e^{-\frac{\lambda}{2} k\tau},
		\end{aligned}
	\end{equation*}
and similarly
	\begin{equation*}
		\begin{aligned}
			\Big|\mathbb{E}\Big(\Big\langle H^b_{t+k\tau},X^{x,\Theta,v}_{t+k\tau}\Big\rangle-\Big\langle H^b_{t+k\tau},X^{\xi^{\Theta,v},\Theta,v}_{t+k\tau}\Big\rangle\Big)\Big|
			%=&\Big|\mathbb{E}\Big\langle K_{t+k\tau}^{-\frac{1}{2}}H^b_{t+k\tau},K_{t+k\tau}^{\frac{1}{2}}\widehat{X}^{\Theta,v}_{t+k\tau}\Big\rangle\Big|\\
			%\le&c\Big(\mathbb{E}\Big\langle K_{t+k\tau}\widehat{X}^{\Theta,v}_{t+k\tau},\widehat{X}^{\Theta,v}_{t+k\tau}\Big\rangle\Big)^{\frac{1}{2}}\\
			%\le&\beta\Big(\mathbb{E}\Big\langle K_t\widehat{X}^{\Theta,v}_t,\widehat{X}^{\Theta,v}_t\Big\rangle e^{-\lambda k\tau}\Big)^{\frac{1}{2}}\\
			%\le&c(1+\mathbb{E}|x-\xi^{\Theta,v}|^2)^{\frac{1}{2}}e^{-\frac{\lambda}{2} k\tau}\\
			\le&c(1+|x|^2+\mathbb{E}|\xi^{\Theta,v}|^2)e^{-\frac{\lambda}{2} k\tau}.
		\end{aligned}
	\end{equation*}
    Hence as $N\rightarrow\infty$, we have
    \begin{equation*}
    	\begin{aligned} &\Big|\mathbb{E}\frac{1}{N}\sum_{k=0}^{N-1}F(t+k\tau,X^{x,\Theta,v}_{t+k\tau},u^{x,\Theta,v}_{t+k\tau})-\mathbb{E}F(t,X^{\xi^{\Theta,v},\Theta,v}_t,u^{\xi^{\Theta,v},\Theta,v}_t)\Big|\\
    	%=&\Big|\frac{1}{N}\sum_{k=0}^{N-1}\mathbb{E}\Big(F(t+k\tau,X^{x,\Theta,v}_{t+k\tau},u^{x,\Theta,v}_{t+k\tau})-\theta_{k\tau}\circ F(t,X^{\xi^{\Theta,v},\Theta,v}_t,u^{\xi^{\Theta,v},\Theta,v}_t)\Big)\Big|\\
    	=&\Big|\frac{1}{N}\sum_{k=0}^{N-1}\mathbb{E}\Big(F(t+k\tau,X^{x,\Theta,v}_{t+k\tau},u^{x,\Theta,v}_{t+k\tau})-F(t+k\tau,X^{\xi^{\Theta,v},\Theta,v}_{t+k\tau},u^{\xi^{\Theta,v},\Theta,v}_{t+k\tau})\Big)\Big|\\
   	    =&\Big|\frac{1}{N}\sum_{k=0}^{N-1}\mathbb{E}\Big(\Big\langle H^a_{t+k\tau}X^{x,\Theta,v}_{t+k\tau},X^{x,\Theta,v}_{t+k\tau}\Big\rangle-\Big\langle H^a_{t+k\tau}X^{\xi^{\Theta,v},\Theta,v}_{t+k\tau},X^{\xi^{\Theta,v},\Theta,v}_{t+k\tau}\Big\rangle\\
	    &\ \ \ \ \ \ \ \ \ \ \ \ \ \ +\Big\langle H^b_{t+k\tau},X^{x,\Theta,v}_{t+k\tau}\Big\rangle-\Big\langle H^b_{t+k\tau},X^{\xi^{\Theta,v},\Theta,v}_{t+k\tau}\Big\rangle\Big)\Big|\\
   		\le&\frac{1}{N}\sum_{k=0}^{N-1}c(1+|x|^2+\mathbb{E}|\xi^{\Theta,v}|^2)e^{-\frac{\lambda}{2} k\tau}\le\frac{c}{N}(1+|x|^2+\mathbb{E}|\xi^{\Theta,v}|^2)(1-e^{-\frac{\lambda}{2}\tau})^{-1}\to0.
    	\end{aligned}
    \end{equation*}
	Then it follows from the dominated convergence theorem that, as $N\rightarrow\infty$,
	\begin{equation*}
		\begin{aligned}		\frac{1}{N\tau}\mathbb{E}\sum_{k=0}^{N-1}\int_{0}^{\tau}F(t+k\tau,X^{x,\Theta,v}_{t+k\tau},u^{x,\Theta,v}_{t+k\tau})\d t\rightarrow\frac{1}{\tau}\mathbb{E}\int_0^\tau F(t,X^{\xi^{\Theta,v},\Theta,v}_t,u^{\xi^{\Theta,v},\Theta,v}_t)\d t.
		\end{aligned}
	\end{equation*}
Therefore, Theorem \ref{transffered form by random periodic solution} follows as the limit $T\rightarrow\infty$ is taken in (\ref{wz17}).
\end{proof}

Based on Corollary \ref{Riccati posedness} and Proposition \ref{linear one order BSDE equation}, we get the random periodic solution to another BSDE which plays a role in completing the square to get the optimal solution to our concerned control problem.
\begin{prop}\label{Riccati BSDE} Assume {\textbf{(A1)}}--{\textbf{(A3)}} and that $(K,L)$ is the $\tau$-random periodic solution to \eqref{Riccati equation}. Then the linear BSDE
	\begin{equation}\label{Riccati and one-order BSDE}
		\begin{aligned}
			\d \eta_t=&-\Big((A_t-B_tR_t^{-1}(B_t^\top K_t+S_t))^\top\eta_t+C_t^\top\zeta_t+K_tb_t+C_t^\top K_t\sigma_t+L_t\sigma_t+q_t\\
			&\ \ \ \ \ -(B_t^\top K_t+S_t)^\top R_t^{-1}\rho_t\Big)\d t+\zeta_t\d W_t
		\end{aligned}
	\end{equation}
admits a unique $\tau$-random periodic solution $(\eta,\zeta)\in S_{\mathbb{F}}^2(0,\infty;\mathbb{S}^n)\times L_{\mathbb{F}}^{2,loc}(0,\infty;\mathbb{S}^n)$.
\end{prop}

Now let's see that the equivalent form of cost functional derived in Theorem \ref{transffered form by random periodic solution} provides a straightforward way to investigate the ergodic control problem over a single periodic interval.
\begin{prop}\label{transffered form by BSDE} Assume {\textbf{(A1)}}--{\textbf{(A3)}} and that  $X^{\xi^{\Theta,v},\Theta,v}$, $(\Theta,v)\in\mathbb{U}$, is the solution to the state equation \eqref{wz2} with initial state $\xi^{\Theta,v}\in L_{\mathscr{F}}^2(\Omega;\mathbb{R}^n)$, where $\xi^{\Theta,v}$ is given in Proposition \ref{wz31}, and $(K,L)$ and $(\eta,\zeta)$ are solutions to \eqref{Riccati equation} and \eqref{Riccati and one-order BSDE}, respectively. Then
\begin{equation*}
		\begin{aligned}
			\mathcal{E}(x,u^{x,\Theta,v})
			=&\frac{1}{\tau}\mathbb{E}\int_{0}^{\tau}\Big(\Big\langle R_t\Big(\Theta_t+R_t^{-1}(B_t^\top K_t+S_t)\Big)X^{\xi^{\Theta,v},\Theta,v}_t,\Big(\Theta_t+R_t^{-1}(B_t^\top K_t+S_t)\Big)X^{\xi^{\Theta,v},\Theta,v}_t\Big\rangle\\
			&\ \ \ \ \ \ \ \ \ \ +2\Big\langle\Big(\Theta_t+R_t^{-1}(B_t^\top K_t+S_t)\Big)X^{\xi^{\Theta,v},\Theta,v}_t,R_tv_t+B_t^\top\eta_t+\rho_t\Big\rangle+\Big\langle R_tv_t,v_t\Big\rangle\\
			&\ \ \ \ \ \ \ \ \ \ +2\Big\langle v_t,B_t^\top\eta_t+\rho_t\Big\rangle+\Big\langle K_t\sigma_t,\sigma_t\Big\rangle+2\Big\langle\eta_t,b_t\Big\rangle+2\Big\langle\zeta_t,\sigma_t\Big\rangle\Big)\d t.
		\end{aligned}
	\end{equation*}
\end{prop}
\begin{proof}
	To save space, we write $X^{\xi^{\Theta,v},\Theta,v}_t$ as $X_t$ in this proof. By Proposition \ref{wz31}, it yields that
	\begin{equation*}
		\begin{aligned}
			\mathbb{E}\Big(\Big\langle K_\tau X_\tau,X_\tau\Big\rangle+2\Big\langle\eta_\tau,X_\tau\Big\rangle\Big)
			=&\mathbb{E}\Big(\Big\langle(\theta_\tau\circ K_0)(\theta_\tau\circ \xi^{\Theta,v}),\theta_\tau\circ \xi^{\Theta,v}\Big\rangle+2\Big\langle\theta_\tau\circ\eta_0,\theta_\tau\circ \xi^{\Theta,v}\Big\rangle\Big)\\
			=&\mathbb{E}\Big(\Big\langle K_0\xi^{\Theta,v},\xi^{\Theta,v}\Big\rangle+2\Big\langle\eta_0,\xi^{\Theta,v}\Big\rangle\Big).
		\end{aligned}
	\end{equation*}
On the other hand, it follows from Theorem \ref{transffered form by random periodic solution} that
	\begin{equation*}
		\begin{aligned}
			\tau\mathcal{E}(x,u^{x,\Theta,v})
			%=&\mathbb{E}\int_0^\tau F(t,X_t,u_t)d t\\
			=&\mathbb{E}\int_0^\tau\Big(\Big\langle(Q_t+S_t^\top\Theta_t+\Theta_t^\top S_t+\Theta_t^\top R_t\Theta_t)X_t,X_t\Big\rangle+2\Big\langle(S_t+R_t\Theta_t)X_t,v_t\Big\rangle\\
			&\ \ \ \ \ \ \ \ \ +2\Big\langle X_t,q_t+\Theta_t^\top\rho_t\Big\rangle+\Big\langle R_tv_t,v_t\Big\rangle+2\Big\langle\rho_t,v_t\Big\rangle\Big)\d t.
		\end{aligned}
	\end{equation*}
Hence applying It\^{o} formula to $\Big\langle K_tX_t,X_t\Big\rangle+2\Big\langle\eta_t,X_t\Big\rangle$, we have
\begin{equation}\label{wz18}
		\begin{aligned}
			0=&\mathbb{E}\int_{0}^{\tau}\d\Big(\left\langle K_tX_t,X_t\right\rangle+2\left\langle\eta_t,X_t\right\rangle\Big)\\
			%=&\mathbb{E}\int_{0}^{\tau}\Big(-\Big\langle\Big(K_tA_t+A_t^\top K_t+C_t^\top %K_tC_t+L_tC_t+C_t^\top L_t+Q_t\\
			%&-(B_t^\top K_t+S_t)^\top R_t^{-1}(B_t^\top K_t+S_t)\Big)X_t,X_t\Big\rangle\\
			%&+\Big\langle K_tX_t,\Big((A_t+B_t\Theta_t)X_t+B_tv_t+b_t\Big)\Big\rangle\\
			%&+\Big\langle K_t\Big((A_t+B_t\Theta_t)X_t+B_tv_t+b_t\Big),X_t\Big\rangle+\Big\langle %K_t\big(C_tX_t+\sigma_t\big),\big(C_tX_t+\sigma_t\big)\Big\rangle\\
			%&+\Big\langle L_tX_t,\big(C_tX_t+\sigma_t\big)\Big\rangle+\Big\langle %L_t\big(C_tX_t+\sigma_t\big),X_t\Big\rangle\\
			%&-2\Big\langle\Big((A_t-B_tR_t^{-1}(B_t^\top %K_t+S_t))^\top\eta_t+C_t^\top\zeta_t+K_tb_t+C_t^\top K_t\sigma_t+L_t\sigma_t+q_t\\
			%&-(B_t^\top K_t+S_t)^\top %R_t^{-1}\rho_t\Big),X_t\Big\rangle+2\Big\langle\eta_t,(A_t+B_t\Theta_t)X_t+B_tv_t+b_t\Big\rangle\\
			%&+2\Big\langle\zeta_t,C_tX_t+\sigma_t\Big\rangle\Big)\d t\\
			=&\mathbb{E}\int_{0}^{\tau}\Big(-\Big\langle\Big(K_tA_t+A_t^\top K_t+C_t^\top K_tC_t+L_tC_t+C_t^\top L_t+Q_t\\
			&\ \ \ \ \ \ \ \ \ \ \ \ \ \ -(B_t^\top K_t+S_t)^\top R_t^{-1}(B_t^\top K_t+S_t)\Big)X_t,X_t\Big\rangle\\
			&\ \ \ \ \ \ \ \ \ +2\Big\langle K_t\big((A_t+B_t\Theta_t)X_t+B_tv_t+b_t\big),X_t\Big\rangle+\Big\langle K_t\big(C_tX_t+\sigma_t\big),\big(C_tX_t+\sigma_t\big)\Big\rangle\\
			&\ \ \ \ \ \ \ \ \ -2\Big\langle\Big((A_t-B_tR_t^{-1}(B_t^\top K_t+S_t))^\top\eta_t+C_t^\top\zeta_t+K_tb_t+C_t^\top K_t\sigma_t+L_t\sigma_t+q_t\\
&\ \ \ \ \ \ \ \ \ \ \ \ \ \ \ -(B_t^\top K_t+S_t)^\top R_t^{-1}\rho_t\Big),X_t\Big\rangle+2\Big\langle L_t\big(C_tX_t+\sigma_t\big),X_t\Big\rangle\\
			&\ \ \ \ \ \ \ \ \ +2\Big\langle\eta_t,(A_t+B_t\Theta_t)X_t+B_tv_t+b_t\Big\rangle+2\Big\langle\zeta_t,C_tX_t+\sigma_t\Big\rangle\\
			&\ \ \ \ \ \ \ \ \ +\Big\langle(Q_t+S_t^\top\Theta_t+\Theta_t^\top S_t+\Theta_t^\top R_t\Theta_t)X_t,X_t\Big\rangle+\Big\langle R_tv_t,v_t\Big\rangle\\
			&\ \ \ \ \ \ \ \ \ +2\Big\langle(S_t+R_t\Theta_t)X_t,v_t\Big\rangle+2\Big\langle X_t,q_t+\Theta_t^\top\rho_t\Big\rangle+2\Big\langle\rho_t,v_t\Big\rangle\Big)\d t-\tau\mathcal{E}(x,u^{x,\Theta,v})
		\end{aligned}
	\end{equation}
\begin{equation*}
		\begin{aligned}			
=&\mathbb{E}\int_{0}^{\tau}\Big(\Big\langle R_t\Big(\Theta_t+R_t^{-1}(B_t^\top K_t+S_t)\Big)X_t,\Big(\Theta_t+R_t^{-1}(B_t^\top K_t+S_t)\Big)X_t\Big\rangle\\
			&\ \ \ \ \ \ \ \ \ +2\Big\langle\Big(\Theta_t+R_t^{-1}(B_t^\top K_t+S_t)\Big)X_t,R_tv_t+B_t^\top\eta_t+\rho_t\Big\rangle+\Big\langle R_tv_t,v_t\Big\rangle\\
			&\ \ \ \ \ \ \ \ \ +2\Big\langle v_t,B_t^\top\eta_t+\rho_t\Big\rangle+\Big\langle K_t\sigma_t,\sigma_t\Big\rangle+2\Big\langle\eta_t,b_t\Big\rangle+2\Big\langle\zeta_t,\sigma_t\Big\rangle\Big)\d t-\tau\mathcal{E}(x,u^{x,\Theta,v}).
		\end{aligned}
	\end{equation*}
Dividing by $\tau$ on both sides of \eqref{wz18}, we arrive at the conclusion.
\end{proof}

We are ready to give an explicit form of optimal control to \textbf{Problem (CL-LQE)}.
\begin{thm} Assume {\textbf{(A1)}}--{\textbf{(A3)}} and that $(K,L)$ and $(\eta,\zeta)$ are solutions to \eqref{Riccati equation} and \eqref{Riccati and one-order BSDE}, respectively. Then
	{\textbf{Problem (CL-LQE)}} is closed-loop solvable with the optimal controls
	\begin{equation}\label{wz32}
		\begin{aligned}
			u^{x,\bar{\Theta},\bar{v}}_t=\bar{\Theta}_t X^{x,\bar{\Theta},\bar{v}}_t+\bar{v}_t,
		\end{aligned}
	\end{equation}
	where $(\bar{\Theta},\bar{v})\in\mathbb{U}$ and $\xi^{\bar{\Theta},\bar{v}}\in L_{\mathscr{F}}^2(\Omega;\mathbb{R}^n)$ given in Proposition \ref{wz31} satisfy for $t\ge0$,
	\begin{equation}\label{wz24}
		\begin{aligned}
			(\bar{\Theta}_t-\Theta^0_t)X^{\xi^{\bar{\Theta},\bar{v}},\bar{\Theta},\bar{v}}_t+\bar{v}_t-v^0_t=0 \ \ {\text a.s.}
		\end{aligned}
	\end{equation}
	with
	\begin{equation*}
		\begin{aligned}
			\Theta^0_t=-R_t^{-1}(B_t^\top K_t+S_t)\ \ {\text and}\ \ v^0_t=-R_t^{-1}(B_t^\top\eta_t+\rho_t).
		\end{aligned}
	\end{equation*}
	In particular, an optimal control is
	\begin{equation}\label{wz25}
		\begin{aligned}		
			u^{x,\Theta^0,v^0}_t=\Theta^0_t X^{x,\Theta^0,v^0}_t+v^0_t,
		\end{aligned}
	\end{equation}
	and the value function in this case is presented by
	\begin{equation}\label{wz27}
		\begin{aligned}
			V=\frac{1}{\tau}\mathbb{E}\int_0^\tau\Big(-\Big\langle R_t^{-1}(B_t^\top\eta_t+\rho_t),B_t^\top\eta_t+\rho_t\Big\rangle+\Big\langle K_t\sigma_t,\sigma_t\Big\rangle+2\Big\langle\eta_t,b_t\Big\rangle+2\Big\langle\zeta_t,\sigma_t\Big\rangle\Big)\d t.
		\end{aligned}
	\end{equation}
\end{thm}
\begin{proof}
For $(\Theta,v)\in\mathbb{U}$, $\xi^{\Theta,v}$ is given in Proposition \ref{wz31}, and we write $X^{\xi^{\Theta,v},\Theta,v}_t$ as $X_t$ in this proof to save space.
	By Proposition \ref{transffered form by BSDE},
\begin{equation}\label{wz26}
		\begin{aligned}
			&\mathcal{E}(x,u^{x,\Theta,v})\\
			=&\frac{1}{\tau}\int_{0}^{\tau}\mathbb{E}\Big(\Big\langle\Big(\Theta_t+R_t^{-1}(B_t^\top K_t+S_t)\Big)R_t\Big(\Theta_t+R_t^{-1}(B_t^\top K_t+S_t)\Big)X_t,X_t\Big\rangle\\
			&\ \ \ \ \ \ \ \ \ \ \ +2\Big\langle\Big(\Theta_t+R_t^{-1}(B_t^\top K_t+S_t)\Big)X_t,R_tv_t+B_t^\top\eta_t+\rho_t\Big\rangle+\Big\langle R_tv_t,v_t\Big\rangle\\
			&\ \ \ \ \ \ \ \ \ \ \ +2\Big\langle v_t,B_t^\top\eta_t+\rho_t\Big\rangle+\Big\langle K_t\sigma_t,\sigma_t\Big\rangle+2\Big\langle\eta_t,b_t\Big\rangle+2\Big\langle\zeta_t,\sigma_t\Big\rangle\Big)\d t\\
			=&\frac{1}{\tau}\int_{0}^{\tau}\mathbb{E}\Big(\Big\langle R_t\Big((\Theta_t+R_t^{-1}(B_t^\top K_t+S_t))X_t+v_t\Big),\Big(\Theta_t+R_t^{-1}(B_t^\top K_t+S_t)\Big)X_t+v_t\Big\rangle\\
			&\ \ \ \ \ \ \ \ \ \ \ +2\Big\langle\Big(\Theta_t+R_t^{-1}(B_t^\top K_t+S_t)\Big)X_t+v_t,B_t^\top\eta_t+\rho_t\Big\rangle+\Big\langle K_t\sigma_t,\sigma_t\Big\rangle\\
			&\ \ \ \ \ \ \ \ \ \ \ +2\Big\langle\eta_t,b_t\Big\rangle+2\Big\langle\zeta_t,\sigma_t\Big\rangle\Big)\d t\\
			=&\frac{1}{\tau}\int_{0}^{\tau}\mathbb{E}\Big(\Big\langle R_t\Big((\Theta_t+R_t^{-1}(B_t^\top K_t+S_t))X_t+v_t+R_t^{-1}(B_t^\top\eta_t+\rho_t)\Big),\\
			&\ \ \ \ \ \ \ \ \ \ \ \ \ \ (\Theta_t+R_t^{-1}(B_t^\top K_t+S_t))X_t+v_t+R_t^{-1}(B_t^\top\eta_t+\rho_t)\Big\rangle\\
			&\ \ \ \ \ \ \ \ \ \ \ -\Big\langle R_t^{-1}(B_t^\top\eta_t+\rho_t),(B_t^\top\eta_t+\rho_t)\Big\rangle+\Big\langle K_t\sigma_t,\sigma_t\Big\rangle+2\Big\langle\eta_t,b_t\Big\rangle+2\Big\langle\zeta_t,\sigma_t\Big\rangle\Big)\d t\\
			\ge&\frac{1}{\tau}\mathbb{E}\int_0^\tau\Big(-\Big\langle R_t^{-1}(B_t^\top\eta_t+\rho_t),B_t^\top\eta_t+\rho_t\Big\rangle+\Big\langle K_t\sigma_t,\sigma_t\Big\rangle+2\Big\langle\eta_t,b_t\Big\rangle+2\Big\langle\zeta_t,\sigma_t\Big\rangle\Big)\d t.
		\end{aligned}
	\end{equation}
	Obviously, the equality holds in above if and only if $(\bar{\Theta},\bar{v})\in\mathbb{U}$ satisfies for $t\in[0,\tau)$,
	\begin{equation*}
		\begin{aligned}
				&\mathbb{E}\int_{0}^{\tau}\Big\langle R_t\Big((\bar{\Theta}_t+R_t^{-1}(B_t^\top K_t+S_t))X^{\xi^{\bar{\Theta},\bar{v}},\bar{\Theta},\bar{v}}_t+\bar{v}_t+R_t^{-1}(B_t^\top\eta_t+\rho_t)\Big),\\
			&\ \ \ \ \ \ \ \ \ \ (\bar{\Theta}_t+R_t^{-1}(B_t^\top K_t+S_t))X^{\xi^{\bar{\Theta},\bar{v}},\bar{\Theta},\bar{v}}_t+\bar{v}_t+R_t^{-1}(B_t^\top\eta_t+\rho_t)\Big\rangle\d t=0,
		\end{aligned}
	\end{equation*}
	which implies \eqref{wz24} and gives the explicit form of the closed-loop optimal controls (\ref{wz32}) for $t\in[0,\tau)$.
	
	In particular, \eqref{wz25} gives
	an optimal control since $(\Theta^0,v^0)\in\mathbb{U}$ by Corollary \ref{Riccati posedness}. With $(\Theta^0,v^0)$, the minimum value in \eqref{wz26} is obtained as same as \eqref{wz27}.
	
	%For any $t=k\tau+r$, where $r\in[0,\tau)$, $k\in\mathbb{N}$.
Since $(\Theta,v)$, $(K,L)$, $(\eta,\zeta)$ and coefficients in \textbf{Problem (CL-LQE)} appearing in above argument are all $\tau$-random periodic, all the results obtained for $t\in[0,\tau)$ still holds for $t\in[0,\infty)$.
\end{proof}
\begin{rmk} We observe from (\ref{wz27}) that the value function  does not depend on the initial time $0$ in  \textbf{Problem (CL-LQE)}. So if the initial times in the state equation and the cost functional are any $t_0>0$ rather than $0$, the value function remains unchanged.
\end{rmk}

\end{document}